\documentclass[notitlepage,11pt]{article}

\usepackage{akshay}

\newcommand{\x}[1][x]{#1}
\newcommand{\simplex}{\mathcal{S}}
\newcommand{\mk}[1][k]{m(\x^{#1})}
\renewcommand{\Sigma}{\mathfrak{S}}

\renewcommand{\geq}{\ge}
\renewcommand{\leq}{\le}

\renewcommand{\akshay}[1]{#1}

\title{Polyhedral Analysis of Symmetric Multilinear Polynomials \\over Box Constraints}

\author{Yibo Xu\affilinfo{Department of Mathematical Sciences, Rensselaer Polytechnic Institute, USA}[xuy24@rpi.edu] \and Warren Adams\affilinfo{Engineering and Information Science Branch (RTA), Air Force Office of Scientific Research, USA}[warren.adams.2@us.af.mil]\and Akshay Gupte\affilinfo{School of Mathematics, University of Edinburgh, UK}[akshay.gupte@ed.ac.uk]}

\date{April 29, 2021}

\begin{document}

\maketitle

\begin{abstract}
It is well-known that the convex and concave envelope of a multilinear polynomial over a box are polyhedral functions. Exponential-sized extended and projected formulations for these envelopes are also known. We consider the convexification question for multilinear polynomials that are symmetric with respect to permutations of variables. Such a permutation-invariant structure naturally implies a quadratic-sized extended formulation for the envelopes through the use of disjunctive programming. The optimization and separation problems are answered directly without using this extension. {The problem symmetry allows the} optimization and separation problems {to be answered} directly without using {any} extension. It also implies that permuting the coefficients of a core set of facets generates all the facets. We provide some necessary conditions and some sufficient conditions for a valid inequality to be a core facet. These conditions are applied to obtain envelopes for two classes: symmetric supermodular functions and multilinear monomials with reflection symmetry, thereby yielding alternate proofs to the literature. Furthermore, we use constructs from the reformulation-linearization-technique to completely characterize the set of points lying on each facet. 

\mykeywords{Convex Hull, Permutation Invariance, Kuhn's Triangulation, Submodular and Supermodular, Reformulation-Linearization Technique}
\mysubjclass{90C23, 90C57, 05E05, 52B15}
\end{abstract}

\section{Introduction}

A polynomial is multilinear if every monomial is square-free in the sense that it is a product of a subset of variables raised to the power one. Multilinear polynomials have degree $n$ and they become linear functions when $n-1$ variables are fixed (hence the name). A multilinear polynomial can be expressed as $\sum_{J\subseteq N}c_{J}\prod_{j\in J}x_{j}$ for some $c\in\real^{2^{n}}$. We are interested in symmetric multilinear polynomials (SMPs) in this paper, by which we mean the polynomial
\begin{subequations}
\begin{equation}
m(\x) \eq \sum_{i=2}^n c_i \sum_{\substack{J \subseteq N\\|J| =i}}\prod_{j \in J}x_j, \label{functionm}
\end{equation}
where we have dropped the constant and linear terms since they are inconsequential to us from the point of view of convexification. The symmetry of $m$ refers to the fact that for every $x\in\real^{n}$ and $\bar{x}$ equal to a permutation of $x$, we have $m(x) = m(\bar{x})$. Our goal is to study the convex hull of the graph of an SMP over a symmetric box in $\real^{n}$. Namely, we study $\conv{G}$ which is the convex hull of the set
\begin{equation}
G = \left\{(\x,y)\in \real^n \times \real: \x \in X, \; y= m(\x)\right\}, \label{Sdef}
\end{equation}
where $X$ is a box imposing the same lower and upper bounds (finite $\ell < u$) on all variables,
\begin{equation}	
X = \left\{\x\in \real^n: \ell \leq x_j \leq u, \ j \in N\right\}.	\label{Xdef}
\end{equation}
\end{subequations}
Throughout this paper, we use $N:= \{1, \ldots, n\}$. \akshay{Although coordinatewise scaling and translation does not break the symmetry of $m(x)$ and can reduce the box $X$ to be the unit hypercube $[0,1]^{n}$, we work with arbitrary $\ell$ and $u$ to avoid this affine transformation which can be cumbersome to perform.

The graph $G_{p} = \{(\x,y)\in B\times\real \colon y = p(x)\}$ of a general multilinear polynomial $p(x) = \sum_{J\subseteq N}c_{J}\prod_{j\in J}x_{j}$ over an arbitrary box $B\subset\real^{n}$ appears not only as a substructure in some important applications but also when optimizing a polynomial over binary variables \citep{del2016polyhedral}, which is equivalent to pseudo-Boolean optimization \citep{boros2002pseudo}, and as an intermediate set when performing factorable reformulations of general polynomial optimization problems \citep{bao2015global,del2019impact,buchheim2008efficient}. Hence, convexification of $G_{p}$ has been the subject of many studies in the literature. It is known that this convex hull is a polytope and exponential-sized extended formulations are available \cite{rikun1997convex,sherali1997convex,ballerstein2014extended} but an explicit description in the $(\x,y)$-space is not known in general. Earlier studies focused on  convexifying multilinear monomials \cite{cafieri2010convex,meyer2004trilinear,ryoo2001analysis,benson2004concave}, motivated by the classical linearization of a monomial $x_{1}x_{2}\dots x_{n}$ over $[0,1]^{n}$ \citep{glover74lin} which defines the convex hull for the graph of the monomial \cite{crama1993concave} and leads to the standard linearization for $G_{p}$. Separation over $G_{p}$ is NP-hard, and so different classes of valid inequalities and cutting plane procedures have been developed for use in global optimization algorithms  \citep{bao2015global,del2019impact,fomeni2015cutting,crama2017class}. There have also been many recent studies on describing $G_{p}$ or generalizations of it in the monomial space, which is obtained by adding a new variable for each monomial, under different assumptions on the structure of the polynomial \cite{del2018multilinear,del2018running,del2016polyhedral,gupte2020bilinear,chen2020multilinear,fischer2018matroid,fischer2020matroid,luedtke2012strength,buchheim2019berge}.

Symmetry has not been exploited in the rich body of literature on convexifying multilinear polynomials. The main objective of this paper is to initiate a systematic polyhedral analysis of $\conv{G}$ by exploiting symmetry in this set. Our focus is on the minimal inequality description of this full-dimensional polytope in the original $(x,y)$-space, as opposed to the many studies in the literature about convexifying general multilinear polynomials in the monomial space. Note that there will be exponentially many monomials in $m(x)$ when $c_{i}\neq 0$ where $i=n/k$ for some constant $k\ge 1$, and so a reformulation to the monomial space will not always be tractable. The symmetric structure we assume is interesting not only because it enables a thorough analysis of the convex hull and hence adds to the convexification literature, but also because it arises in many applications of combinatorial optimization \cite{anthony2016quadratization,boros2019compact,dey2020spca,kim2019convexification} and with regards to chromatic number of graphs and other areas of combinatorics \citep[cf.][]{eagles2020h,stanley1995symmetric}. 
}

\subsection{Our Contributions}

There are exponentially many facets, but symmetry of the function $m$ and of the set $X$ means that we only need to focus on a certain subset of inequalities, which we call core inequalities, and all other inequalities are generated as permutations of coefficients in core inequalities. In theory, all these inequalities can be obtained by projecting an extended formulation having  $O(n^{2})$ many variables and $O(n)$ many constraints, which is much smaller in size than the exponential-sized extensions for general multilinear polynomials. However, projecting this extension is a tedious task due to the combinatorial explosion that generally occurs with the projection operation. Instead, we give several necessary conditions and some sufficient conditions for a core valid inequality to be facet-defining, and these are more tractable to verify than the conditions that come from the use of polarity since the latter require enumeration of extreme points of a polyhedron. They can be applied to certify whether a given description of $\conv{G}$ is minimal or not. In that regard, we use our conditions to obtain explicit listing of all the facet-defining inequalities for two families of SMPs. The first family is that of supermodular SMPs and the second family is that of monomials whose lower and upper bounds are reflections of each other ($-\ell = u > 0$).

The Reformulation-Linearization Technique (RLT) is known to convexify the graph of a general multilinear polynomial over an arbitrary box. We show that RLT also implies that such a polynomial is nonnegative on a box if and only if it is nonnegative at every vertex of the box. This consequence enables us to characterize the set of points on $G$ that lie on a facet of the convex hull of $G$. The sets, by derivation, turn out to be different unions of $d $-dimensional ($0 \le d\le n-1$) faces of $\conv{G}.$

The questions of optimization and separation are also answered for $\conv{G}$ for any SMP without using the quadratic-sized extension. A linear function can be optimized over $\conv{G}$ in $O(n\log{n})$ time (assuming the value of function $m$ at a vertex of $X$ can be computed in $O(1)$ time) without using the extended formulation. Thus, a point can be separated from $\conv{G}$ in polynomial time via the ellipsoid method. There is also a direct separation algorithm that runs in $O(nt + n\log{n})$ time where $t$ is the number of core inequalities, and hence has polynomial time complexity when $t$ is bounded by a polynomial in input size.

\subsubsection{Related Results in Literature}

\akshay{Although we recognize that our convex hull descriptions for the two special families have been established before in literature, our necessary conditions for core facets of $\conv{G}$ yield alternate proofs for them. In this context, the following results are known. A set function is supermodular if and only if a certain extension of it from the vertices of a box to the entire box is a concave function \cite[Proposition 4.1]{lovasz1983submod}; see \cite[Theorem 4]{iwata2008submodular} for a simple proof. The projection of this extension generates the concave envelope of a supermodular function, and this envelope is described in the $x$-space by so-called polymatroid inequalities; see \cite[Theorem 1]{atamturk2008polymatroids}, \cite[Lemma 4]{vondrak2010lecture}, \cite[Theorem 3.3]{tawarmalani2013explicit} for some proofs of this well-known fact. These polymatroid inequalities can be derived from polarity and the results on optimizing over the polymatroid polyhedron \cite{edmonds1970submodular}. There is one inequality for each $n$-permutation, allowing for repetitions, and this immediately relates to there being a single core facet of the envelope and all other facets being permutations of it. 
The convex envelope of a supermodular function is not known in general and is NP-hard to separate over. For the symmetric structure that is considered in this paper, this envelope was first described by \cite[Theorem 4.6]{tawarmalani2013explicit} using their approach of strategically computing an exponential number of subdivisions of $X$ and by using the extreme points of these subdivisions to calculate the facets. 
A similarity between this proof and ours is that they both rely on the Kuhn's triangulation \cite{kuhn1960lemma} of a box. The convex hull of a monomial with $-\ell = u$, was first established by the authors in \citep[Theorem 4.1]{monomerr}. However, it was done so without using the symmetry in the monomial to recognize the core facets. 
}

Since the original submission of this paper, convexification of permutation-invariant sets, such as the graph $G$, has been studied by \cite{kim2019convexification} who give a general framework for obtaining an extended formulation with $O(n^{2})$ many variables and inequalities for the convex hulls of such sets. This framework is based on first convexifying a strategically-defined subset of the region of interest, and then obtaining the convex hull for all permutations of each point within the convexified set. In the realm of convex hulls for SMPs, \cite{kim2019convexification} and this paper both utilize information relative to a specific simplex that generates the Kuhn's triangulation of a box, and this similarity is not entirely surprising because of symmetry in the problem.

\subsection{Organisation of the Paper}
Section~2 presents an extended formulation for $\conv{G}$, gives a polar description of this convex hull, introduces the concept of core inequalities from which all valid linear inequalities can be derived upto permutations the coefficients, and answers the questions of optimization and separation. Section~3 provides various conditions for a core valid equality to be a core facet. Section~4 gives the RLT theory for general multilinear polynomials and derives consequences of it on nonnegativity of the polynomial over a box. This RLT machinery enables us to characterize the set of points in the graph at which a core facet is satisfied exactly. Section~5 considers the case of $m(\x)$ being a supermodular function and Section~6 considers the case of $m(\x)$ being a monomial with the variable domain being a box that allows reflections across the origin. Finally, Section~7 provides a summary of the paper and highlights some outstanding open questions for future research. The Appendix gives alternate arguments for deriving basic properties of $\conv{G}$ using the RLT, and also has illustrative examples for our main results.

\section{Preliminaries}

\akshay{The convex hull of $G$ is a full-dimensional set because $G$ is the surface of a nonlinear function taken over an $n$-dimensional box, and this set is a polytope whose vertices are in bijection to the vertices of the box $X$ since this is known for general multilinear functions \citep{rikun1997convex,sherali1997convex}.} Throughout this paper, we will study inequalities of the form
\begin{equation}
\beta_0+\sum_{j=1}^n\beta_jx_j+\beta^{\prime}y \geq 0. \label{sum24}
\end{equation}
\akshay{Trivial facets of $\conv{G}$ (also referred to as vertical facets in the literature) are the facets generated by valid $\x[\beta]$ with $\beta^{\prime}=0$. For $n\ge 3$, trivial facets are precisely the bounds on the variables, and for $n=2$ there are no trivial facets \citep[Theorem 2.4 and Remark 2.1]{bao2009multiterm}. Wlog and upto scaling, we can assume that every nontrivial valid inequality has $\beta^{\prime}=\pm{1}$. Since $\conv{G}$ is the intersection of the epigraph of the convex envelope of $m$ and hypograph of the concave envelope of $m$, a facet-defining inequality (facet) with $\beta^{\prime}=1$ (resp. $\beta^{\prime}=-1$) represents a nontrivial facet of the epigraph (resp. hypograph). Since $\conv{G}$ is full-dimensional, for every nontrivial facet of $\conv{G}$ there exists a unique nontrivial valid inequality. We will characterize nontrivial facets using the concept of core inequalities.

Certain notation will be useful throughout our study. Our polyhedral analysis will rely on the simplex
\begin{equation}
\simplex \equiv \left\{\x\in \real^n: u \geq x_1\geq\ldots\geq x_n\geq \ell \right\}, \label{simpleex}
\end{equation}
whose $n+1$ extreme points are
\begin{equation}	\label{def:xk}
\x^{k} \equiv \left(\underbrace{u, \ldots, u}_{k},\, \underbrace{\ell,\ldots,\ell}_{n-k} \right), \quad k= 0,\dots,n.
\end{equation}
\akshay{The simplex $\simplex$ is the one that has been used in combinatorial geometry to yield the Kuhn's triangulation of a box  \cite{kuhn1960lemma}. It is also known to be useful for convexifying general submodular/supermodular functions \cite{lovasz1983submod,tawarmalani2013explicit} and we will see this also in \textsection\ref{sec:supermod}. The value of the multilinear function at each $\x^{k}$ can be computed by substituting it into equation~\eqref{functionm}. In the special case of $\ell=0$ and $u=1$, this becomes the combinatorial formula $m(\x^{k}) = c_{k} + \sum_{j=2}^{k-1}\binom{k}{j}\,c_{j}$.} For notational convenience, we let ${\cal L}_k(\beta_0,\x[\beta])= \beta_0+\sum_{j=1}^n\beta_jx^k_j$ be that value obtained by inserting extreme point $\x^k$ of $\simplex$ into the expression $\beta_0+\sum_{j=1}^n\beta_jx_j$ of \eqref{sum24}, where $x^k_j$ denotes entry $j$ of $\x^k,$ so that 
\begin{equation}
{\cal L}_k(\beta_0,\x[\beta])=\beta_0+u\left(\sum_{j=1}^k\beta_j\right)+\ell \left(\sum_{j=k+1}^n\beta_j\right), \qquad \forall \; k \in \{0, \ldots, n\}, \label{anoto25}
\end{equation} 
with $\sum_{j=1}^0\beta_j=0$ and $\sum_{j=n+1}^n\beta_j=0$ in ${\cal L}_0(\beta_0,\x[\beta])$ and ${\cal L}_n(\beta_0,\x[\beta]),$ respectively. 

We begin by noting two implicit descriptions of $\conv{G}$, one is an extended formulation that projects onto this convex hull and another is a polarity result that gives a characterization of all the facets. Then we introduce core inequalities as those inequalities having a nondecreasing order on the coefficients and permutations of which generate all the valid inequalities. Lastly, we give algorithms for optimizing and separating over $\conv{G}$.
}

\subsection{Implicit Descriptions of the Convex Hull}

\akshay{A straightforward application of disjunctive programming along with using the well-known result that the envelopes of a multilinear function are generated by the extreme points of the box gives us a $O(n^{2})$-sized extended formulation for $\conv{G}$.

\begin{proposition}	\thlabel{extform}
The following polyhedron projects onto $\conv{G}$,
\begin{align*}
\Big\{(x,\{w^{k}\}_{k=0}^{n}, y , v, \lambda) \sep & x_{j} = \sum_{k=0}^{n}w^{k}_{j}, \; \forall j\in N, \ y = \sum_{k=0}^{n}v_{k}, \ \sum_{k=0}^{n}\lambda_{k}=1, \\ 
& \sum_{j=1}^{n}w^{k}_{j} = (ku + (n-k)\ell)\lambda_{k}, \ v_{k} = \mk\lambda_{k}, \ k=0,\dots,n \\
& \ell \lambda_{k} \le w^{k}_{j} \le u\lambda_{k} \; \forall j,k, \ v \in \real^{n+1}, \ \lambda \in \real^{n+1}_{+} \Big\}
\end{align*}
\end{proposition}
\begin{proof}
It is well-known \cite{sherali1997convex,rikun1997convex} that envelopes of a multilinear function over a box have the vertex extendability property meaning that for any multilinear function $p(x)$ and box $B\subset\real^{n}$, the convex hull of $\{(x,y)\in B\times\real\colon y = p(x) \}$ is equal to the convex hull of $\{(x,y)\in \vertex{B}\times\real\colon y = p(x) \}$. Therefore, we have $\conv{G}$ being equal to $\conv{\{(x,y)\in\{\ell,u\}^{n}\times\real\colon y = m(x)\}}$. The vertex set $\{\ell,u\}^{n}$ can be partitioned into $n+1$ sets with each set for $k\in\{0,\dots,n\}$ corresponding to points having $k$ entries of $u$, thereby giving us
\begin{equation}	\label{convunion}
\vertex{(\conv{G})} \eq \bigcup_{k=0}^{n}\bigcup_{\sigma\in\Sigma_{n}} \left\{(\sigma\cdot\x^{k},m(\sigma\cdot\x^{k})) \right\}
\eq \bigcup_{k=0}^{n} \left\{ Q_{k}  \times \{\mk\} \right\},
\end{equation}
where $\Sigma_{n}$ is the symmetric group of $n$ elements, $\sigma\cdot v = (v_{\sigma(1)}, v_{\sigma(2)}, \dots, v_{\sigma(n)})$ for a vector $v\in\real^{n}$, $Q_{k} := \cup_{\sigma\in\Sigma_{n}}\{\sigma\cdot\x^{k}\}$,  and the second equality is due to symmetry of $m(x)$. Since $\conv{G}$ is the convex hull of its vertices, we get $\conv{G} = \conv{\left(\bigcup_{k=0}^{n} \left\{ \conv{Q_{k}}  \times \{\mk\} \right\}\right)}$. Since $Q_{k}$ is the set of all points that have $k$ entries of $u$ and $n-k$ entries of $\ell$, we have $\conv{Q_{k}} = \{x\in[\ell,u]^{n}\colon \sum_{j=1}^{n}x_{j} = ku + (n-k)\ell \}$. Set $P_{k} := \conv{Q_{k}}\times\{\mk\} = \{(x,y)\in[\ell,u]^{n}\times\real\colon \sum_{j=1}^{n}x_{j} = ku + (n-k)\ell, \, y = \mk\}$, which is a polyhedron. Since $\conv{G} = \conv(\cup_{k=0}^{n}P_{k})$, applying disjunctive programming \cite{balas1979disjunctive} yields the desired extended formulation.
\end{proof}

Another way of implicitly describing the convex hull is to characterise all its facets using the extreme points of the polar of the convex hull. This is known from \cite[Theorem 2]{sherali1997convex} and \cite[Theorem 2.4]{bao2009multiterm} for general multilinear functions over arbitrary boxes, and hence we state the below result for an SMP without proof.

\begin{proposition}
An inequality~\eqref{sum24} with $\beta' = 1$ (resp. $\beta' = -1$) is a facet of $\conv{G}$ if and only if $(\beta_{0},\beta)$ is an extreme point of the polyhedron $E := \{(\beta_{0},\beta)\colon \beta_{0}+\beta^{\top}\bar{\x} \ge - m(\bar{\x}), \ \bar{x}\in\{\ell,u\}^{n} \}$ (resp. $H := \{(\beta_{0},\beta)\colon \beta_{0}+\beta^{\top}\bar{\x} \ge m(\bar{\x}), \ \bar{x}\in\{\ell,u\}^{n} \}$).
\end{proposition}

The challenge with using this polarity result is that it requires enumeration of extreme points of the set $E$ and $H$, which are exponentially many in general. The next section observes that it suffices to focus attention on only a subset of inequalities since all other inequalities are obtained via permutations.
}

\subsection{Core Inequalities}

We define \emph{core inequalities} as being those inequalities~\eqref{sum24} that have $\beta^{\prime}=\pm{1}$ and $\beta_1 \leq \ldots \leq \beta_n.$ The restriction that $\beta^{\prime}=\pm{1}$ for nonzero $\beta^{\prime}$ is nonrestrictive, as it follows from scaling. The restriction that $\beta_1 \leq \ldots \leq \beta_n$ follows from the problem symmetry, as an inequality will be valid (a facet) for $\conv{G}$ if and only if every symmetric copy obtained by permuting the entries in $\x[\beta]^{\top}$ is also valid (a facet). A \emph{valid core inequality} is a core inequality that is valid for $\conv{G}.$ A valid core inequality that is also a facet for $\conv{G}$ is a \emph{core facet}. 

Symmetry implies that the validity of a core inequality \eqref{sum24} can be checked in terms of only the $(n+1)$ extreme points of the simplex $\simplex.$ 

\begin{lemma}	\thlabel{foundation}
A core inequality \eqref{sum24} is a valid core inequality if and only if 
\begin{equation}
{\cal L}_k(\beta_0,\x[\beta]) + \beta^{\prime} m(\x^k) \geq 0, \qquad \forall \; k \in \{0, \ldots, n\}. \label{anoto}
\end{equation}
\end{lemma}
\begin{proof}
The only if direction is trivial and so we consider the if direction. It is sufficient to show that the $(n+1)$ inequalities of \eqref{anoto} imply that \eqref{sum24} is nonnegative for all $2^n$ extreme points of $\conv{G}.$ Toward this end, consider any $k \in\{0, \ldots, n\},$ and note that each of the $n \choose k$ extreme points $(\x,y)$ of $\conv{G}$ with $\x$ having $k$ entries of value $u$ and $(n-k)$ entries of value $\ell$ also has $y = m(\x^k).$ In addition, each such extreme point yields a value for $\left(\beta_0+\sum_{j=1}^n\beta_jx_j\right)$ that is at least as large as ${\cal L}_k(\beta_0,\x[\beta]),$ so that \eqref{sum24} is nonnegative for all $n \choose k$ such extreme points. As the result holds true for every such $k,$ it holds true for all the extreme points of $\conv{G}.$ This completes the proof.
\end{proof}

\subsection{Optimization and Separation}

\akshay{
The question of optimizing a linear function over $\conv{G}$ can be solved via a sorting algorithm  without using the quadratic-sized extended formulation of \mythref{extform}.

\begin{proposition}	\thlabel{optcompl}
For any $(\alpha,\alpha')\in\real^{n}\times\real$,
\[
\max\left\{\alpha^{\top}x + \alpha^{\prime}y \sep (x,y)\in\conv{G} \right\} \eq \max_{k=0,\dots,n}\left\{u\sum_{j=1}^{k}\alpha_{\sigma(j)} + \ell\sum_{j=k+1}^{n}\alpha_{\sigma(j)} + \alpha^{\prime}\mk \right\},
\]
where the permutation $\sigma$ is such that $\alpha_{\sigma(1)}\ge \alpha_{\sigma(2)} \ge \cdots \ge \alpha_{\sigma(n)}$.
\end{proposition}
\begin{proof}
The linear function has an optimum at a vertex of $\conv{G}$, and so the left-hand side is equivalent to maximising $\alpha^{\top}x + \alpha^{\prime}y$ over $\vertex{(\conv{G})}$. Our claim follows after using $\vertex{(\conv{G})} = \cup_{k=0}^{n}\{Q_{k}\times\{\mk\} \}$ from \eqref{convunion} and observing that $\max\{\alpha^{\top}x \colon x\in Q_{k}\} = u\sum_{j=1}^{k}\alpha_{\sigma(j)} + \ell\sum_{j=k+1}^{n}\alpha_{\sigma(j)}$.
\end{proof}

Hence, assuming computation of each $\mk$ takes $O(1)$ time, linear optimization over $\conv{G}$ can be solved in $O(n\log{n})$ time by sorting the vector $\alpha$.

Due to the well-known equivalence of complexity of optimization and separation over a polyhedron, it follows that a given point can be separated from the convex hull of $G$ in polynomial time. However, this connection invokes the ellipsoid algorithm whose complexity is a high degree polynomial in the input encoding. The question of separation can be answered directly if there are polynomially many core facets. The approach is similar to that used for \mythref{optcompl} but here the point $(\bar{x},\bar{y})$ to be separated has $\bar{x}$ sorted in nonincreasing order. 

\begin{proposition}	\thlabel{sepcompl}
If $\conv{G}$ has $t$ core facets, then a point can be separated from $\conv{G}$ in $O(nt+n\log{n})$ time.
\end{proposition}
\begin{proof}
Let $(\bar{x},\bar{y})\in X\times\real$ be a given point. Take $\sigma\in\Sigma_{n}$ such that $\bar{x}_{\sigma(1)} \le \bar{x}_{\sigma(2)} \le \cdots \le \bar{x}_{\sigma(n)}$. Choose any core facet~\eqref{sum24}. We have $\beta_{1}\le \cdots\le\beta_{n}$ and  need to check if this facet or any permuted facet is violated, i.e., whether there exists a permutation $\sigma\in\Sigma_{n}$ such that $\beta_0+\sum_{j=1}^n\beta_{\sigma(j)} \bar{x}_j+\beta^{\prime}\bar{y} < 0$. This is equivalent to checking whether $- \beta_{0} - \beta^{\prime}\bar{y} > \min\{\bar{x}^{\top}v\colon v\in P_{n}(\beta) \}$, where $P_{n}(\beta) = \conv{\{(\beta_{\sigma(1)},\beta_{\sigma(2)},\dots,\beta_{\sigma(n)}) \colon \sigma\in\Sigma_{n} \}}$ is the generalized permutahedron \cite{bowman1972permutation} with respect to $\beta$. Linear optimization over the permutahedron can be done via the sorting algorithm since the permutahedron is the base polymatroid polyhedron corresponding to a certain submodular function \cite{rado52ineq} and \citet{edmonds1970submodular} established that the sorting algorithm works for optimizing over any base polymatroid. In particular, we have that our minimum is attained at $v^{*} = \sigma^{-1}\cdot\hat{\beta}$, where $\hat{\beta} := (\beta_{n},\beta_{n-1},\dots,\beta_{1})$ is the reverse sorting of $\beta$, so that $v^{*}_{j} = \beta_{n+1-\sigma^{-1}(j)}$ for all $j\in N$. Thus, checking separation of a single core facet and its permuted copies is equivalent to checking whether $-\beta_{0} - \beta^{\prime}\bar{y} > \sum_{j=1}^{n}\beta_{n+1-\sigma^{-1}(j)}\bar{x}_{j}$, which runs in $O(n)$ time. The overall complexity then becomes $O(nt+n\log{n})$ due to the sorting of $\bar{x}$ initially and enumeration over $t$ core facets.
\end{proof}

As seen in the above proof, the overall complexity is composed of $nt$ additions and multiplications and a sorting step which requires $O(n\log{n})$ comparisons.
}

\section{Conditions for Core Facets}

A challenge in determining whether a valid core inequality is a core facet is the identification of a maximum number of affinely independent points within $\conv{G}$ that satisfies the inequality exactly. Of course, we can restrict attention to only those points that are extreme to $\conv{G}.$ The following proposition shows that the set of extreme points of $\conv{G}$ which satisfy a core  inequality exactly can be completely characterized in terms of the extreme points of the simplex $\simplex$.

\begin{lemma}	\thlabel{foundation45}
Given an extreme point $(\tilde{\x},m(\tilde{\x}))$ of $\conv{G}$ with $\tilde{\x}$ not an extreme point of $\simplex$, let $r$ be the smallest index $j$ such that $x_j =\ell,$ and $s$ be the largest index $j$ such that $x_j =u.$ Further let $p$ be the number of entries of $\tilde{\x}$ having value $u$ (so that $r \leq p \leq s-1$). Then $(\x,y)=(\tilde{\x},m(\tilde{\x}))$ satisfies a valid core inequality \eqref{sum24} exactly if and only if $(\x,y)=(\x^p,m(\x^p))$ satisfies the inequality exactly, and $\beta_r = \ldots = \beta_s.$
\end{lemma}
\begin{proof}
We have $\beta_0+\sum_{j=1}^n\beta_j\tilde{x}_j+\beta^{\prime}m(\tilde{\x}) \geq {\cal L}_p(\beta_0,\x[\beta])+\beta^{\prime}m(\x^p) \geq 0$, where the first inequality is due to the nondecreasing values of $\beta_j$ and the symmetry of $m(\x),$ and the second inequality is due to the validity of \eqref{sum24}. The first inequality is satisfied exactly if and only if $\beta_r = \ldots = \beta_s,$ while the second inequality is satisfied exactly if and only if \eqref{sum24} is satisfied exactly at $(\x^p,m(\x^p)).$
\end{proof}

A consequence is that given a valid core inequality, we can identify the subset of $2^n$ extreme points of $\conv{G}$ that satisfy the inequality exactly by considering only the $(n+1)$ extreme points of $\simplex.$ The next few results build upon this consequence to provide characteristics of core facets in terms of the extreme points of $\simplex.$

\begin{proposition}	\thlabel{foundation22}
A valid core inequality \eqref{sum24} is a core facet only if it is satisfied exactly at $(\x^k,m(\x^k))$ for at least two extreme points $\x^k$ of the simplex $\simplex,$ with at least one extreme  point not being $\x^0$ or $\x^n.$
\end{proposition}
\begin{proof}
Suppose that a valid core inequality \eqref{sum24} is satisfied exactly at fewer than two such points $(\x^k,m(\x^k)),$ or at only the points $(\x^0,m(\x^0))$ and $(\x^n,m(\x^n)).$ Since the dimension of $\conv{G}$ is $n+1,$ it is sufficient to show that the inequality is satisfied exactly at no more than $n$ affinely independent extreme points of $\conv{G}.$ Three cases arise. First, if the inequality is satisfied exactly at no such $(\x^k,m(\x^k)),$ then \eqref{sum24} is not satisfied exactly at any of the $2^n$ extreme points of $\conv{G}$ by \mythref{foundation45}. Second, if the inequality is satisfied exactly at precisely one such $(\x^{\tilde{k}},m(\x^{\tilde{k}}))$ then, by \mythref{foundation45}, the only extreme points $(\x,m(\x))$ of $\conv{G}$ that can possibly satisfy \eqref{sum24} exactly are the $n \choose \tilde{k}$ points with $\x$ having $\tilde{k}$ entries of value $u$ and $(n-\tilde{k})$ entries of value $\ell.$ However, as the two linearly independent hyperplanes $\sum_{j=1}^nx_j=\tilde{k}u+(n-\tilde{k})\ell$ and $y=m(\x^{\tilde{k}})$ both pass through these $n \choose \tilde{k}$ points, there exist at most $n$ such affinely independent points. Finally, if the inequality is satisfied exactly at only $(\x^0,m(\x^0))$ and $(\x^n,m(\x^n)),$ then \mythref{foundation45} gives us that these are the only extreme points of $\conv{G}$ that satisfy \eqref{sum24} exactly.
\end{proof}

The necessary condition in  \mythref{foundation22} to be a core facet is also sufficient when all the coefficients are equal.

\begin{proposition}	\thlabel{equalcoeff}
A valid core inequality \eqref{sum24} with $\beta_1 = \ldots = \beta_n$ is a core facet if and only if it is satisfied exactly at $(\x^k,m(\x^k))$ for at least two extreme points $\x^k$ of the simplex $\simplex,$ with at least one extreme point not being $\x^0$ or $\x^n.$
\end{proposition}
\begin{proof}
The only if direction follows directly from \mythref{foundation22}, and so we consider the if direction. Suppose the inequality is satisfied exactly at two such points $(\x^p,m(\x^p))$ and $(\x^q,m(\x^q)),$ with $\x^p$ not being $\x^0$ or $\x^n.$ Then \mythref{foundation45} gives us that the $n \choose p$ extreme points $(\x,m(\x))$ of $\conv{G}$ with $\x$ having $p$ entries of value $u$ and $(n-p)$ entries of value $\ell$ satisfy the inequality exactly. The proof is to show that there exist $n$ affinely independent points from amongst this set of $n \choose p$ points. Then these $n$ points, together with $(\x^q,m(\x^q)),$ will form an affinely independent set of $(n+1)$ points because each of the first $n$ points satisfies the equation $\sum_{j=1}^nx_j=pu+(n-p)\ell,$ but $(\x^q,m(\x^q))$ does not. In fact, it is sufficient to show that the extreme points of $X$ associated with these $n \choose p$ points are affinely independent.

Since the affine independence of a collection of points is unaffected when the same value is subtracted from every entry of each point, and when each point is multiplied by a nonzero scalar, the affine independence of the associated $n \choose p$ extreme points of $X$ remains unchanged when every $u$ is replaced with 1 and every $\ell$ is replaced with $0.$ Consider the $n \times {{n} \choose {p}}$ matrix $A$ defined so that each column corresponds to one such point, upon application of these operations. If $p=1,$ then $A$ is a permutation matrix, and the $n \choose p$ points are affinely independent. Otherwise, $A$ has the two properties that: each row contains $\kappa_1 \equiv  {{n-1} \choose {p-1}}$ entries of value 1, and every pair of two distinct rows contains $\kappa_2 \equiv  {{n-2} \choose {p-2}}$ common entries of value $1.$ Hence, $AA^T$ is the $n \times n$ matrix having $\kappa_1$ along the main diagonal and $\kappa_2$ elsewhere. Since $\mbox{rank}(AA^T)=n,$ because the common row sum allows us to subtract $\kappa_2$ from every entry of $AA^T$ to obtain a lower-triangular matrix with $(\kappa_1-\kappa_2) \neq 0$ along the diagonal, then $\mbox{rank}(A)=n,$ and the proof is complete.
\end{proof}

The below proposition gives further conditions, in terms of the extreme points $\x^k$ of $\simplex,$ for a valid core inequality to be a core facet.

\begin{proposition}	\thlabel{foundation24}
Given any two valid core inequalities of the form 
\begin{equation}
\bar{\beta}_0+\sum_{j=1}^n\bar{\beta}_jx_j+\beta^{\prime}y \geq 0 \mbox{ and } \hat{\beta}_0+\sum_{j=1}^n\hat{\beta}_jx_j+\beta^{\prime}y \geq 0, \label{firsti}
\end{equation}
if 
\begin{equation}
{\cal L}_k(\bar{\beta}_0,\bar{\x[\beta]}) \leq {\cal L}_k(\hat{\beta}_0,\hat{\x[\beta]}) \; \forall \; k \in \{0, \ldots, n\}, \label{secondi} 
\end{equation}
with strict inequality holding for at least one $k$ in \eqref{secondi}, then the right inequality of \eqref{firsti} is not a facet.
\end{proposition}
\begin{proof}
It is sufficient to show that every extreme point $(\x,m(\x))$ of $\conv{G}$ that satisfies the right inequality of \eqref{firsti} exactly also satisfies the left inequality of \eqref{firsti} exactly. This statement holds for the $(n+1)$ extreme points $(\x^k,m(\x^k)), k \in \{0, \ldots, n\},$ by \eqref{secondi}, and
so we arbitrarily select any one of the remaining $2^n-(n+1)$ extreme points, say $(\tilde{\x},m(\tilde{\x})),$ and suppose that the right inequality holds exactly at this point. Some $p \in \{1, \ldots, n-1\}$ entries of $\tilde{\x}$ have value $u$ and $(n-p)$ entries have value $\ell,$ with the first $p$ entries not all equal to $u.$ 
The proof reduces to showing that 
\begin{equation}
0={\cal L}_p(\hat{\beta}_0,\hat{\x[\beta]})+\beta^{\prime}m(\x^p)={\cal L}_p(\bar{\beta}_0,\bar{\x[\beta]})+\beta^{\prime}m(\x^p)=\bar{\beta}_0+\sum_{j=1}^n\bar{\beta}_j\tilde{x}_j+\beta^{\prime}m(\tilde{\x}).
\label{summary}
\end{equation}
Each equality of \eqref{summary} is considered separately.
\begin{itemize}
\item Since, by assumption, $(\tilde{\x},m(\tilde{\x}))$ satisfies the right inequality of \eqref{firsti} exactly, \mythref{foundation45} gives us that the first equality of \eqref{summary} holds true, and also that $\hat{\beta}^*=\hat{\beta}_r = \ldots =\hat{\beta}_s$ for some scalar $\hat{\beta}^*,$ where $r$ and $s$ are, respectively, the indices of the first and last entries of $\tilde{\x}$ which differ from $\x^p.$
\item The second equality of \eqref{summary} holds true by \eqref{secondi} with $k=p,$ the first equality of \eqref{summary}, and the validity of the left inequality of \eqref{firsti} at $(\x^p,m(\x^p)).$ 
\item
 Since as noted above, $\hat{\beta}^*=\hat{\beta}_r = \ldots =\hat{\beta}_s$ for some scalar $\hat{\beta}^*,$ we have that
\begin{equation}
{\cal L}_p(\hat{\beta}_0,\hat{\x[\beta]})+\hat{\beta}^*(u-\ell)(s-p)={\cal L}_s(\hat{\beta}_0,\hat{\x[\beta]}) \geq {\cal L}_s(\bar{\beta}_0,\bar{\x[\beta]})={\cal L}_p(\bar{\beta}_0,\bar{\x[\beta]})+(u-\ell)(\sum_{j=p+1}^s\bar{\beta}_j) \label{balance3}
\end{equation}
and
\begin{equation}
{\cal L}_p(\hat{\beta}_0,\hat{\x[\beta]})-\hat{\beta}^*(u-\ell)(p-r+1)={\cal L}_{r-1}(\hat{\beta}_0,\hat{\x[\beta]}) \geq {\cal L}_{r-1}(\bar{\beta}_0,\bar{\x[\beta]})={\cal L}_p(\bar{\beta}_0,\bar{\x[\beta]})-(u-\ell)(\sum_{j=r}^p\bar{\beta}_j), \label{balance4}
\end{equation}
where, within \eqref{balance3} and \eqref{balance4}, the equalities follow from \eqref{anoto25}, and the inequalities are due to \eqref{secondi}. Combine these expressions and invoke ${\cal L}_p(\hat{\beta}_0,\hat{\x[\beta]})={\cal L}_p(\bar{\beta}_0,\bar{\x[\beta]})$ from the second equality of \eqref{summary} to obtain 
\begin{equation}
\hat{\beta}^* \leq \left(\frac{1}{p-r+1}\right)\left(\sum_{j=r}^p\bar{\beta}_j\right) \leq \left(\frac{1}{s-p}\right)\left(\sum_{j=p+1}^s\bar{\beta}_j\right) \leq \hat{\beta}^*, \nonumber
\end{equation}
where the three inequalities follow from \eqref{balance4}, the nondecreasing property of $\bar{\x[\beta]},$ and \eqref{balance3}, respectively. Again by the nondecreasing property of $\bar{\x[\beta]},$ we have that $\hat{\beta}^*=\bar{\beta}_r = \ldots =\bar{\beta}_s,$ giving the third equality of \eqref{summary} by \mythref{foundation45}.\qedhere
\end{itemize}
\end{proof}

This leads us to the following necessary and sufficient condition for a valid inequality with distinct coefficients to be a core facet.

\begin{corollary}\label{sfirst}
A valid inequality $\bar{\beta}_0+\sum_{j=1}^n\bar{\beta}_jx_j+\beta^{\prime}y \geq 0$ with $\bar{\beta}_1 < \ldots < \bar{\beta}_n$ is a core facet if and only if it is satisfied exactly at $(\x^k,m(\x^k))$ for $k \in\{0, \ldots, n\}$. In this case, no other core facet can exist with the given $\beta^{\prime}.$
\end{corollary}
\begin{proof}
The if direction follows from the $n+1$ points $(\x^k,m(\x^k))$ for $k \in\{0, \ldots, n\}$ being affinely independent, and so we consider the only if direction. \mythref{foundation45} gives us that the only extreme points to $\conv{G}$ that can possibly satisfy the given inequality exactly are $(\x^k,m(\x^k))$ for $k \in\{0, \ldots, n\}.$ Since the inequality is a facet, these $(n+1)$ affinely independent points must satisfy the inequality exactly, giving ${\cal L}_k(\bar{\beta}_0,\bar{\x[\beta]})+\beta^{\prime}m(\x^k)=0$ for $k \in\{0, \ldots, n\}.$ No other valid inequality $\hat{\beta}_0+\sum_{j=1}^n\hat{\beta}_jx_j+\beta^{\prime}y \geq 0$ for $\conv{G}$ with the given $\beta^{\prime}$ can then be a facet by \mythref{foundation24}.
\end{proof}

Another set of necessary conditions for being a core facet is obtained below.

\begin{proposition}	\thlabel{equalcoeff7}
Given any two valid core inequalities of the form \eqref{firsti}, suppose that the following conditions hold.
\begin{enumerate}
\item $\bar{\beta}_u=\bar{\beta}_v$ for each $(u,v),u<v,$ having $\hat{\beta}_u=\hat{\beta}_v.$  
\item $\{k:{\cal L}_k(\hat{\beta}_0,\hat{\x[\beta]})+\beta^{\prime}m(\x^k)=0\} \subset  \{k:{\cal L}_k(\bar{\beta}_0,\bar{\x[\beta]})+\beta^{\prime}m(\x^k)=0\}.$
\end{enumerate}
Then the right inequality of \eqref{firsti} is not a facet.
\end{proposition}
\begin{proof}
As with the proof of \mythref{foundation24}, it is sufficient to show that every extreme point of $\conv{G}$ that satisfies the right inequality of \eqref{firsti} exactly also satisfies the left inequality of \eqref{firsti} exactly. This statement holds for the $(n+1)$ extreme points $(\x^k,m(\x^k)), k \in\{0, \ldots, n\},$ by Condition 2 above, and
so we arbitrarily select any one of the remaining $2^n-(n+1)$ extreme points, say $(\tilde{\x},m(\tilde{\x})),$ and suppose that the right inequality holds exactly at this point. Some $p \in \{1, \ldots, n-1\}$ entries of $\tilde{\x}$ have value $u$ and $(n-p)$ entries have value $\ell,$ with the first $p$ entries not all equal to $u.$ By \mythref{foundation45}, we have that $\hat{\beta}_r = \ldots = \hat{\beta}_s$ where $r$ and $s$ are, respectively, the indices of the first and last entries of $\tilde{\x}$ which differ from $\x^p.$ The first condition above then gives $\bar{\beta}_r = \ldots = \bar{\beta}_s,$ so that by again invoking \mythref{foundation45}, we have the left inequality of \eqref{firsti} is satisfied exactly at $(\tilde{\x},m(\tilde{\x})).$
\end{proof}

The below result identifies instances in which select coefficients of a core facet must have equal values. We do not need to consider $r=n$ or $s=0$ because such cases are not possible by \mythref{foundation22}. 

\begin{proposition} \thlabel{equalcoeff8}
Given a core facet \eqref{sum24}, let 
\[
\begin{split}
r&=\min\left\{k \in \{0, \ldots, n-1\}\colon {\cal L}_k(\beta_0,\x[\beta]) +\beta^{\prime} m(\x^k)=0\right\}, \\ 
s&=\max\left\{k \in \{1, \ldots, n\}\colon{\cal L}_k(\beta_0,\x[\beta]) +\beta^{\prime} m(\x^k)=0\right\}. 
\end{split}
\]
Then $\beta_1 = \ldots = \beta_{r+1}$ and $\beta_{s} = \ldots = \beta_n.$
\end{proposition}
\begin{proof}
We consider separately the conclusions that $\beta_1 = \ldots = \beta_{r+1}$ and $\beta_{s} = \ldots = \beta_n.$
\begin{itemize}
\item If $r=0,$ the first conclusion follows trivially. Otherwise, $r \geq 1$ and, by contradiction, define $p \leq r$ so that $\beta_p < \beta_{p+1} = \beta_{r+1}.$ It is sufficient to show that there exists an $\epsilon>0$  so that the inequality
\begin{equation}
(\beta_0-pu\epsilon)+\sum_{j=1}^{p}(\beta_j+\epsilon)x_j+\sum_{j=p+1}^n\beta_jx_j+\beta^{\prime}y \geq 0, \label{lastly}
\end{equation}
with $(\beta_p+\epsilon) \leq \beta_{p+1},$ is valid for $\conv{G}.$ Then every extreme point $(\x,m(\x))$ to $\conv{G}$ with $x_j=u$ for $j \in \{1, \ldots, p\}$ will have the left side of \eqref{lastly} equal to the left side of \eqref{sum24}, while every extreme point $(\x,m(\x))$ to $\conv{G}$ with $x_j=\ell$ for at least one $j \in \{1, \ldots, p\}$ will have the left side of \eqref{lastly} strictly less than the left side of \eqref{sum24}. Relative to the extreme points of the simplex $\simplex$, the left sides of \eqref{lastly} and \eqref{sum24} will both take value ${\cal L}_k(\beta_0,\x[\beta]) +\beta^{\prime} m(\x^k)$ at $(\x^k,m(\x^k))$ for each $k \in \{p, \ldots, n\},$ but the left side of \eqref{lastly} will be ${\cal L}_k(\beta_0,\x[\beta]) +\beta^{\prime} m(\x^k)-(p-k)(u-\ell)\epsilon$ at $(\x^k,m(\x^k))$ for each $\x^k, k \in \{0, \ldots, p-1\},$ which is $(p-k)(u-\ell)\epsilon>0$ less than the left side of \eqref{sum24}. Define $\epsilon=\min\{\epsilon^{\prime},(\beta_{p+1}-\beta_{p})\},$ where $\epsilon^{\prime}=\mbox{min}_{k \in \{0, \ldots, p-1\}}\left\{\frac{{\cal L}_k(\beta_0,\x[\beta]) +\beta^{\prime} m(\x^k)}{(p-k)(u-\ell)}\right\}.$  Then $\epsilon >0$ and inequality \eqref{lastly} is valid for $\conv{G}$ by \mythref{foundation}, as it is valid at all $(\x^k,m(\x^k)),k \in \{0, \ldots, n\}.$ 
\item If $s=n,$ the second conclusion follows trivially. Otherwise $s \leq (n-1)$ and, by contradiction, define $p \geq s$ so that $\beta_{s} = \beta_p < \beta_{p+1}.$  It is sufficient to show that there exists an $\epsilon>0$ so that the inequality
\begin{equation}
\beta_0+(n-p)\ell \epsilon+\sum_{j=1}^{p}\beta_jx_j+\sum_{j=p+1}^n(\beta_j-\epsilon)x_j+\beta^{\prime}y \geq 0, \label{lastly2}
\end{equation}
with $\beta_p \leq (\beta_{p+1}-\epsilon),$ is valid for $\conv{G}.$ Then every extreme point $(\x,m(\x))$ to $\conv{G}$ with $x_j=\ell$ for $j \in \{p+1, \ldots, n\}$ will have the left side of \eqref{lastly2} equal to the left side of \eqref{sum24}, while every extreme point $(\x,m(\x))$ to $\conv{G}$ with $x_j=u$ for at least one $j \in \{p+1, \ldots, n\}$ will have the left side of \eqref{lastly2} strictly less than the left side of \eqref{sum24}. Relative to the extreme points of the simplex $\simplex$, the left sides of \eqref{lastly2} and \eqref{sum24} will both take value ${\cal L}_k(\beta_0,\x[\beta]) +\beta^{\prime} m(\x^k)$ at $(\x^k,m(\x^k))$ for each $k \in \{1, \ldots, p\},$ but the left side of \eqref{lastly2} will be ${\cal L}_k(\beta_0,\x[\beta]) +\beta^{\prime} m(\x^k)-(k-p)(u-\ell)\epsilon$ at $(\x^k,m(\x^k))$ for each $\x^k, k \in \{p+1, \ldots, n\},$ which is $(k-p)(u-\ell)\epsilon>0$ less than the left side of \eqref{sum24}. Define $\epsilon=\min\{\epsilon^{\prime},(\beta_{p+1}-\beta_{p})\},$ where $\epsilon^{\prime}=\mbox{min}_{k \in \{p+1, \ldots, n\}}\left\{\frac{{\cal L}_k(\beta_0,\x[\beta]) +\beta^{\prime} m(\x^k)}{(k-p)(u-\ell)}\right\}.$  Then $\epsilon >0$ and inequality \eqref{lastly2} is valid for $\conv{G}$ by \mythref{foundation}, as it is valid at all $(\x^k,m(\x^k)),k \in \{0, \ldots, n\}.$ \qedhere
\end{itemize}
\end{proof}

\section{RLT for Multilinear Functions}
In this section, we present the Reformulation-Linearization Technique (RLT) as it pertains to the set $G$. We provide a brief description of select aspects of the general RLT process that are relevant to this study, emphasizing some key properties. Then we   review the mathematical details of the RLT in terms of Kronecker products of matrices. This RLT machinery enables us to characterize exactness of the core facets for $\conv{G}$.

\subsection{Main Ideas}
The RLT is a general methodology for reformulating mixed-integer linear and polynomial programs for the purpose of obtaining tight linear programming relaxations. While there is a rich body of literature on the topic \citep{sherali1994hierarchy,sherali1990hierarchy,sherali99rlt} we focus attention here on a box-constrained region of $n$ variables $x_j,$ where each $x_j$ is restricted to lie between variable bounds $L_j$ and $U_j.$  The RLT gives a hierarchy of successively tighter polyhedral relaxations, but we consider only the highest level $n$ which affords the convex hull representations.

Specifically, consider the set
\begin{equation}
X^{\prime} \equiv \left\{\x\in \real^n: L_j \leq x_j \leq U_j \; \forall \; j \in N\right\} \label{Xprimedef}
\end{equation}
having each $L_j <U_j,$ which is a generalization of the set $X$ of \eqref{Xdef} that reduces to \eqref{Xdef} when $L_j=\ell$ and $U_j=u$ for all $j \in N.$ The RLT process that we apply to \eqref{Xprimedef} consists of the two steps of \emph{reformulation} and \emph{linearization}. The \emph{reformulation} step computes products of the expressions $(x_j-L_j)$ and $(U_j-x_j),$ taken $n$ at a time, such that one such expression is chosen for each $j.$ In this manner, $2^n$ multilinear polynomial functions of degree $n$ emerge. To elaborate, define the $2^n$ functions $F_{K}(x)$ so that 
\begin{equation}
F_{K}(x)= \prod_{j\in K}(x_j-L_j)\prod_{j\notin K}(U_j-x_j), \quad K \subseteq N. \label{Fun}
\end{equation}
Then we have $2^n$ multilinear polynomial functions of the form
\begin{equation}
F_{K}(x), \quad K \subseteq N. \label{functions}
\end{equation}
Each of these functions is nonnegative for all $\x \in X^{\prime},$ and the RLT enforces this nonnegativity to obtain the $2^n$ multilinear polynomial inequalities
\begin{equation}
F_{K}(x) \geq 0, \quad K \subseteq N, \label{RLTstep}
\end{equation}
that are satisfied for all $\x \in X^{\prime}.$ 

The \emph{linearization} step then substitutes a continuous variable $w_J$ for each of the $2^n-(n+1)$ distinct product terms $\prod_{j \in J}x_j$ with $J\subseteq N$ and $|J| \geq 2$ that are found in \eqref{RLTstep}. Denote the linearized form of each function $F_{K}(x)$ that is obtained via such substitutions as $F_{K}(x).$ The RLT gives the following \emph{polyhedral} set:
\begin{equation}
P\equiv\left\{(\x,\x[w]) \in \real^n\times\real^{2^n-(n+1)}:F_{K}(x) \geq 0, \ K \subseteq N\right\}.  \label{RLTstep2}
\end{equation}
We adopt the notation that $\{\bullet\}_L$ is the linearized form of the vector $\bullet$ that is obtained by substituting $w_J=\prod_{j \in J}x_j$ for all $J \subseteq N$ with $|J| \geq 2$ throughout $\bullet.$ In this manner, $F_{K}(x)=\{F_{K}(x)\}_L$ for all $K \subseteq N.$ Two properties of the set $P$ are as follows.

\begin{proposition}	\thlabel{propP}
The set $P$ is a polytope having exactly $2^n$ extreme points and at each such extreme point $(\x,\x[w])$ we have
\begin{enumerate}
\item $x_j\in\{L_j,U_j\}$ for all $j \in N,$
\item $w_J=\prod_{j \in J}x_j$ for all $J \subseteq N$ with $|J| \geq 2.$
\end{enumerate}
Consequently, we have $P = \conv{T}$ where 
\begin{equation}
T\equiv\left\{(\x,\x[w]) \in \prod_{j\in N}[L_{j},U_{j}] \, \times\real^{2^n-(n+1)}: w_{J} = \prod_{j \in J}x_j, \ J \subseteq N \mbox{ with }|J| \geq 2\right\}. \label{miss1}
\end{equation}
\end{proposition}

These properties of $P$ were originally established by \citet{sherali1990hierarchy} when $P$ has $L_j=0$ and $U_j=1$ for all $j \in N$; see also \cite[Theorem 1]{fomeni2015cutting}. They were more generally proven in \cite{adams2005hierarchy,henry} using Kronecker products of matrices for richer families of polytopes that subsume $P.$ These same type Kronecker products are also found in the generalization of the RLT constructs to handle general integer variables \cite{adams2011,adams2005hierarchy}. For completeness, we present a tailored version of the arguments in the next subsection. In addition to establishing the two stated properties of $P,$ this Kronecker product representation of \eqref{functions} through \eqref{RLTstep2} motivates two consequences that will be used later to characterize exactness of facet-defining inequalities. These apply to any multilinear polynomial and therefore encompass the functions $m(\x)$.

\begin{corollary} \thlabel{obs1}
For every $p(x) = \sum_{J \subseteq N}\alpha_j\prod_{j \in J}x_j,$ there exists a unique $\x[\pi] \in \real^{2^n}$ so that $p(x) = \sum_{K \subseteq N}\pi_{K}F_{K}(x)$. In particular, $\pi_K=\frac{p(\hat{\x})}{F_{K}(\hat{\x})},$ where $\hat{x}_j=U_j$ for all $j \in K$ and $\hat{x}_j=L_j$ for all $j \notin K.$
\end{corollary}

\begin{corollary} \thlabel{obs2}
For every $p(x) = \sum_{J \subseteq N}\alpha_j\prod_{j \in J}x_j,$ we have $p(x) \ge 0$ for all $x\in X^{\prime}$ if and only if $p(x) \ge 0$ for every extreme point $x\in X^{\prime}$. 
\end{corollary}

\mythref{propP} can be extended to include the variable $y$ and restriction $y= m(\x)$ of \eqref{Sdef}, so that the projection operation from the resulting higher-variable space onto the original $(\x,y)$ space can, in theory, be used to compute $\conv{G}.$ \akshay{These properties were already noted by \citet{sherali1997convex} using a different approach,  
we give our arguments in Appendix~\ref{sec:propG} for sake of completeness.}

\subsection{Use of Kronecker Products} 

\begin{proof}[\textbf{Proof of \mythref{propP}}]
We prove \mythref{propP} by expressing \eqref{functions} and \eqref{RLTstep} as Kronecker products of matrices, and then using a known result of Kronecker products. The consequence $P = \conv{T}$ follows since every extreme point of the polytope $P$ is in $T,$ and since $T \subseteq P$ by construction.

Recall that the Kronecker product of an $m_1 \times n_1$ matrix $A$ with an $m_2 \times n_2$ matrix $B$, denoted by $A \otimes B,$ is that $m_1m_2 \times n_1n_2$ matrix  $A \otimes B = \scriptsize \left[\begin{array}{ccc} a_{11}B & \hdots & a_{1n_1}B \\ \vdots & \ddots & \vdots \\ a_{m_11}B & \hdots & a_{m_1n_1}B \end{array}\right], \normalsize$
\noindent where $a_{ij}$ is the $(i,j)^{th}$ entry of $A.$ A result of Kronecker products is that, given any collection of $r$ pairs of matrices $A_i$ and $B_i$ for $i=1,...,r$ such that the products $A_iB_i$ are defined, we have that 
\begin{equation}
A_1B_1 \otimes A_2B_2 \otimes \ldots \otimes A_rB_r= (A_1 \otimes A_2 \otimes \ldots \otimes A_r)(B_1 \otimes B_2 \otimes \ldots \otimes B_r). \label{kroneck}
\end{equation}
As a consequence, if each of the matrices $A_i$ is invertible with inverse $A_i^{-1},$ then by setting $B_i=A_i^{-1}$ for all $i$ within \eqref{kroneck}, we obtain that
\begin{equation}
(A_1 \otimes A_2 \otimes \ldots \otimes A_{r})^{-1} = A_1^{-1} \otimes A_2^{-1} \otimes \ldots A_r^{-1}, \label{kroneck2}
\end{equation}
since the left side of \eqref{kroneck} becomes the suitably-sized identity matrix. 

Relative to the functions $F_{K}(x)$ of \eqref{functions} and \eqref{RLTstep}, and the functions $F_{K}(x)$ and set $P$ of \eqref{RLTstep2}, define for each variable $x_j,$ the two matrices $\scriptsize \left[\begin{array}{cc} U_j &  -1 \\ -L_j & 1 \end{array}\right] \normalsize$ and $\scriptsize \left[\begin{array}{c} 1 \\ x_j \end{array}\right] \normalsize$ so that the functions $(U_j-x_j)$ and $(x_j-L_j)$ can be represented in matrix form as the two entries of the vector $\scriptsize \left[\begin{array}{cc} U_j &  -1 \\ -L_j & 1 \end{array}\right] \left[\begin{array}{c} 1 \\ x_j \end{array}\right] \normalsize.$ Then the $2^n$ functions of \eqref{functions} can be collectively expressed using Kronecker products as 
\begin{equation}
\scriptsize \left(\left[\begin{array}{cc} U_1 &  -1 \\ -L_1 & 1 \end{array}\right] \left[\begin{array}{c} 1 \\ x_1 \end{array}\right]\right)\normalsize \otimes \ldots \otimes \scriptsize \left(\left[\begin{array}{cc} U_n &  -1 \\ -L_n & 1 \end{array}\right] \left[\begin{array}{c} 1 \\ x_n \end{array}\right]\right)\normalsize, \nonumber
\end{equation}
and the $2^n$ inequalities of \eqref{RLTstep} can be similarly expressed as 
\begin{equation}
\scriptsize \left(\left[\begin{array}{cc} U_1 &  -1 \\ -L_1 & 1 \end{array}\right] \left[\begin{array}{c} 1 \\ x_1 \end{array}\right]\right)\normalsize \otimes \ldots \otimes \scriptsize \left(\left[\begin{array}{cc} U_n &  -1 \\ -L_n & 1 \end{array}\right] \left[\begin{array}{c} 1 \\ x_n \end{array}\right]\right)\normalsize  \geq	 \boldsymbol{0}, \label {prodkron}
\end{equation}
where $\boldsymbol{0}$ is the $2^n$-dimensional column vector of all zeroes. Identity \eqref{kroneck} allows us to rewrite \eqref{prodkron} as
\begin{equation}
\scriptsize \left(\left[\begin{array}{cc} U_1 &  -1 \\ -L_1 & 1 \end{array}\right] \normalsize \otimes \ldots \otimes \scriptsize \left[\begin{array}{cc} U_n &  -1 \\ -L_n & 1 \end{array}\right]\right)\left(\left[\begin{array}{c} 1 \\ x_1 \end{array}\right]\normalsize \otimes \ldots \otimes \scriptsize \left[\begin{array}{c} 1 \\ x_n \end{array}\right]\right) \normalsize \geq \boldsymbol{0}. \label {prodkron2}
\end{equation}

The set $P$ of \eqref{RLTstep2} can then be expressed in terms of \eqref{prodkron2} as 
\begin{alignat}{2} 
P=\mbox{\Huge \{} \normalsize &(\x,\x[w]) \in \real^n\times\real^{2^n-(n+1)}:   \nonumber \\ 
&\scriptsize \left(\left[\begin{array}{cc} U_1 &  -1 \\ -L_1 & 1 \end{array}\right] \normalsize \otimes \ldots \otimes \scriptsize \left[\begin{array}{cc} U_n &  -1 \\ -L_n & 1 \end{array}\right]\right)\left\{\left[\begin{array}{c} 1 \\ x_1 \end{array}\right]\normalsize \otimes \ldots \otimes \scriptsize \left[\begin{array}{c} 1 \\ x_n \end{array}\right]\right\}_L \normalsize \geq \boldsymbol{0}\mbox{\Huge \}} \normalsize, && \label{ca4} 
\end{alignat}
which is equivalent to 
\begin{alignat}{2} 
P=\mbox{\Huge \{} \normalsize &(\x,\x[w]) \in \real^n\times\real^{2^n-(n+1)}:   \nonumber \\ 
&\scriptsize \left(\left[\begin{array}{cc} U_1 &  -1 \\ -L_1 & 1 \end{array}\right] \normalsize \otimes \ldots \otimes \scriptsize \left[\begin{array}{cc} U_n &  -1 \\ -L_n & 1 \end{array}\right]\right)\left\{\left[\begin{array}{c} 1 \\ x_1 \end{array}\right]\normalsize \otimes \ldots \otimes \scriptsize \left[\begin{array}{c} 1 \\ x_n \end{array}\right]\right\}_L \normalsize = \x[\lambda] \nonumber \\ & \mbox{for some nonnegative } \x[\lambda] \in \real^{2^n}\mbox{\Huge \}} \normalsize. && \label{ca5} 
\end{alignat}

For each $j,$ the inverse of the matrix $\scriptsize \left[\begin{array}{cc} U_j &  -1 \\ -L_j & 1 \end{array}\right]$ \normalsize is $\scriptsize \frac{1}{d_j}\left[\begin{array}{cc} 1 &1 \\ L_j & U_j \end{array}\right],$ \normalsize where $d_j=(U_j-L_j)$ denotes the difference between $U_j$ and $L_j.$ Then we apply \eqref{kroneck2} to \eqref{ca5} to rewrite the set $P$ as
\begin{alignat}{2} 
P=\mbox{\Huge \{} \normalsize &(\x,\x[w]) \in \real^n\times\real^{2^n-(n+1)}:  \nonumber \\ 
&\left\{\scriptsize\left[\begin{array}{c} 1 \\ x_1 \end{array}\right]\normalsize \otimes \ldots \otimes \scriptsize \left[\begin{array}{c} 1 \\ x_n \end{array}\right]\right\}_L \normalsize =   \left(\scriptsize \frac{1}{d_1}\left[\begin{array}{cc} 1 &  1 \\ L_1 & U_1 \end{array}\right] \normalsize \otimes \ldots \otimes \scriptsize \frac{1}{d_n}\left[\begin{array}{cc} 1 &  1 \\ L_n & U_n \end{array}\right]\right)\normalsize\x[\lambda] \nonumber \\ & \mbox{for some nonnegative } \x[\lambda] \in \real^{2^n}\mbox{\Huge \}} . && \label{ca6} 
\end{alignat}

The two claimed properties of $P$ become apparent in light of \eqref{ca6}. The first equation of \eqref{ca6} is equivalent via scaling to $D_n=\sum_{j=1}^{2^n}\lambda_j,$ with $D_n \equiv \prod_{j=1}^nd_j,$ thereby establishing $P$ as a polytope having exactly $2^n$ extreme points. These points can be numbered so that extreme point $j$ has $\x[\lambda]$ given by $\lambda_j = D_n$ and $\lambda_i=0$ for $i \neq j,$ and $(\x,\x[w])$ given by column $j$ of the matrix $\scriptsize \left[\begin{array}{cc} 1 &  1 \\ L_1 & U_1 \end{array}\right] \normalsize \otimes \ldots \otimes  \scriptsize \left[\begin{array}{cc} 1 &  1 \\ L_n & U_n \end{array}\right], \normalsize$ less the first row. Each of the $2^n$ columns of this last matrix consists of a distinct realization of the vector $\scriptsize\left[\begin{array}{c} 1 \\ x_1 \end{array}\right]\normalsize \otimes \ldots \otimes \scriptsize \left[\begin{array}{c} 1 \\ x_n \end{array}\right] \normalsize,$ less the first entry, evaluated at some $\x$ having each $x_j$ fixed at either its lower bound $L_j$ or its upper bound $U_j.$ 
\end{proof}

\begin{proof}[\textbf{Proof of \mythref{obs1}}]
Invertibility of the matrix $\left(\scriptsize\left[\begin{array}{cc} U_1 &  -1 \\ -L_1 & 1 \end{array}\right] \normalsize \otimes \ldots \otimes \scriptsize \left[\begin{array}{cc} U_n &  -1 \\ -L_n & 1 \end{array}\right]\right)$ from \eqref{prodkron2} establishes the $2^n$ functions $F_{K}(x)$ of \eqref{functions} as a basis for the vector space consisting of all multilinear polynomials having degree at most $n.$ This space clearly has dimension $2^n,$ and the invertibility demonstrates the functions to be linearly independent. As a result, associated with every multilinear polynomial $\sum_{J \subseteq N}\alpha_j\prod_{j \in J}x_j,$ there exists a unique $\x[\pi] \in \real^{2^n}$ so that $p(x) = \sum_{K \subseteq N}\pi_{K}F_{K}(x)$ holds. This identity can be readily solved in terms of $\pi_K$ due to the structure of the functions $F_{K}(x)$ to obtain the claimed expression for $\pi$.
\end{proof}

\begin{proof}[\textbf{Proof of \mythref{obs2}}]
The multipliers $\pi_K$ of \mythref{obs1} can be used to establish, in terms of the extreme points of the set $X^{\prime}$, the nonnegativity of $\sum_{J \subseteq N}\alpha_J\prod_{j \in J}x_j$ over $X^{\prime},$ and the collection of points at which the polynomial vanishes. Clearly, such a polynomial is nonnegative over $X^{\prime}$ if and only if $\pi_K \ge 0$ for all $K$. \mythref{obs1} uniquely defines each such multiplier to be the positive scalar $\frac{p(\hat{\x})}{F_{K}(\hat{\x})}$ evaluated at a distinct extreme point $\hat{x}$. 
\end{proof}


\subsection{Exactness of Core Inequalities}

In this section, we exploit the RLT properties to characterise the set of points $(\x,y) \in G$ at which each valid core inequality is satisfied exactly. We consider only core inequalities since other inequalities can be handled via permutations of the variables $x_j$ to coincide with the permutations of the coefficients $\beta_j.$ We begin with three propositions and a theorem. The first proposition provides a necessary and sufficient condition for a point $\tilde{\x} \in X$ to satisfy a function $F_{K}(x)$ of \eqref{Fun} exactly, and the second gives three equivalent conditions for exactness to hold. The third proposition uses the RLT results of Observations 1 and 2 as a theoretical bridge to translate the conditions of the second proposition to that of satisfying a valid inequality for $\conv{G}$ exactly. The theorem uses this third proposition to give a necessary and sufficient condition for a point $(\tilde{\x},\tilde{y}) \in G$ to satisfy a valid core inequality exactly. We subsequently present two special cases of the theorem as corollaries. Consistent with our earlier work, we then separately address the core facets when the SMP $m(\x)$ is supermodular over the extreme points of $X$ of \eqref{Xdef}, and when $m(\x)$ is a monomial having $-\ell=u>0.$ The two corollaries serve to simplify this analysis. Throughout, the inequalities $F_{K}(x) \geq 0$ of \eqref{RLTstep} are assumed to have $L_j=\ell$ and $U_j=u$ for all $j \in N$ within \eqref{Fun} as in $X$ of \eqref{Xdef}.   

\begin{lemma}\thlabel{fite000}
Given any $K \subseteq N$ and any $\tilde{\x} \in X$ of \eqref{Xdef}, partition $N$  into $N_1,$ $N_2,$ and $N_3$ so that $N_1 \equiv\{j:\tilde{x}_j=\ell\},$ $N_2 \equiv\{j:\tilde{x}_j=u\},$ and $N_3 \equiv\{j:\ell<\tilde{x}_j<u\}.$ Then $\tilde{\x}$ satisfies $F_{K}(x)=0$ if and only if $\left([K \cap N_1]  \cup [(N\setminus K) \cap N_2]\right) \neq \emptyset.$
\end{lemma}
\begin{proof}
Trivial by the definition of the polynomial $F_{K}(x)$ in \eqref{Fun}.
\end{proof}

\begin{proposition}\thlabel{fite00}
Given any $K \subseteq N$ and any $\tilde{\x} \in X$ of \eqref{Xdef}, partition $N$  into $N_1,$ $N_2,$ and $N_3$ so that $N_1 \equiv\{j:\tilde{x}_j=\ell\},$ $N_2 \equiv\{j:\tilde{x}_j=u\},$ and $N_3 \equiv\{j:\ell<\tilde{x}_j<u\}.$ The following statements are equivalent:
\begin{enumerate}
\item $\tilde{\x}$ satisfies $F_{K}(x)=0,$ 
\item $\hat{\x}$ satisfies $F_{K}(x)=0$ at every $\hat{\x} \in X$ having $\hat{x}_j=\ell \; \forall \; j \in N_1$ and $\hat{x}_j=u\; \forall \; j \in N_2,$
\item $\hat{\x}$ satisfies $F_{K}(x)=0$ at every extreme point $\hat{\x}$ of $X$ having $\hat{x}_j=\ell \; \forall \; j \in N_1$ and $\hat{x}_j=u\; \forall \; j \in N_2.$
\end{enumerate}
\end{proposition}
\begin{proof}
Given any $K \subseteq N$ and any $\tilde{\x} \in X$ with $N$ partitioned into $N_1,$ $N_2,$ and $N_3$ as described, the proof is to show that $1 \rightarrow 2 \rightarrow 3 \rightarrow 1.$ The case $2 \rightarrow 3$ is trivial and so we consider the remaining two cases. For convenience, we define the set $\mathcal{N}_K \subseteq N$ as $\mathcal{N}_K\equiv \left([K \cap N_1]  \cup [(N\setminus K) \cap N_2]\right).$ \\
$\left(1 \rightarrow 2\right)$ 
\mythref{fite000} gives us that $\tilde{\x}$ satisfies $F_{K}(x)=0$ only if $\mathcal{N}_K \neq \emptyset.$ But $\mathcal{N}_K$ remains nonempty for every $\hat{\x} \in X$ having $\hat{x}_j=\ell \; \forall \; j \in N_1$ and $\hat{x}_j=u \; \forall \; j \in N_2$ so that the ``if" condition of \mythref{fite000} establishes the result. \\
$\left(3 \rightarrow 1\right)$ 
Consider that extreme point $\hat{\x}$ of $X$ having $\hat{x}_j=\ell \; \forall \; j \in N_1,$ $\hat{x}_j=u \; \forall \; j \in N_2,$  $\hat{x}_j=\ell \; \forall \; j \in (N_3-K),$ and $\hat{x}_j=u \; \forall \; j \in (K \cap N_3).$ Then 
\begin{equation}
F_{K}(x)=0=(u-\ell)^{(n-|\mathcal{N}_K|)}\left(\prod_{j \in (K \cap N_1)}(\hat{x}_j-\ell)\prod_{j \in (N\setminus K) \cap N_2}(u-\hat{x}_j)\right)\nonumber
\end{equation} 
when evaluated at $\hat{\x},$ implying that $\mathcal{N}_K \neq \emptyset.$ Then the ``if" condition of \mythref{fite000} establishes the result.
\end{proof}

The result below mirrors that of \mythref{fite00}, but invokes Observations 1 and 2 to extend the results from the functions $F_{K}(x)$ to the inequalities \eqref{sum24}.
\begin{proposition}\thlabel{fitt}
Given any inequality \eqref{sum24} that is valid for $\conv{G}$ and any $\tilde{\x} \in X$ of \eqref{Xdef}, partition $N$ into $N_1,$ $N_2,$ and $N_3$ so that $N_1 \equiv\{j:\tilde{x}_j=\ell\},$ $N_2 \equiv\{j:\tilde{x}_j=u\},$ and $N_3 \equiv\{j:\ell<\tilde{x}_j<u\}.$ The following statements are equivalent:
\begin{enumerate}
\item $(\x,y)=(\tilde{\x},m(\tilde{\x}))$ satisfies \eqref{sum24} exactly,
\item $(\x,y)=(\hat{\x},m(\hat{\x}))$ satisfies \eqref{sum24} exactly at every $\hat{\x} \in X$ having $\hat{x}_j=\ell \; \forall \; j \in N_1$ and $\hat{x}_j=u\; \forall \; j \in N_2,$
\item $(\x,y)=(\hat{\x},m(\hat{\x}))$ satisfies \eqref{sum24} exactly at every extreme point $\hat{\x}$ of $X$ having $\hat{x}_j=\ell \; \forall \; j \in N_1$ and $\hat{x}_j=u\; \forall \; j \in N_2.$
\end{enumerate}
\end{proposition}
\begin{proof}
Given any $K \subseteq N$ and any $\tilde{\x} \in X$ with $N$ partitioned into $N_1,$ $N_2,$ and $N_3$ as described, the proof is to show that $1 \rightarrow 2 \rightarrow 3 \rightarrow 1.$ The case $2 \rightarrow 3$ is trivial and so we consider the remaining two cases. To begin, note that Observations 1 and 2 combine to show, given any inequality \eqref{sum24} that is valid for $\conv{G},$ there exists a unique, nonnegative $\x[\pi] \in \real^{2^n}$ satisfying $\beta_0+\sum_{j=1}^n\beta_jx_j+\beta^{\prime}m(\x)= \prod_{K \subseteq N}\pi_KF_{K}(x)$ so that, given any $\bar{\x} \in X,$
\begin{equation}
\beta_0+\sum_{j=1}^n\beta_j\bar{x}_j+\beta^{\prime}m(\bar{\x})= 0 \mbox{ if and only if every } \pi_K>0 \mbox{ has } F_{K}(x)=0 \mbox{ at } \bar{\x}. \label{newbeed}
\end{equation}
$\left(1 \rightarrow 2\right)$ Given $(\x,y)=(\tilde{\x},m(\tilde{\x}))$ satisfies \eqref{sum24} exactly, the ``only if" direction of \eqref{newbeed} with $\bar{\x}=\tilde{\x}$ gives us that every $\pi_K>0$ has $F_{K}(x)=0$ at $\tilde{\x}.$ Implication $1 \rightarrow 2$ of \mythref{fite00} then gives us that every $\pi_K>0$ has $F_{K}(x)=0$ at every $\hat{\x} \in X$ having $\hat{x}_j=\ell \; \forall \; j \in N_1$ and $\hat{x}_j=u\; \forall \; j \in N_2.$ Then the ``if" direction of \eqref{newbeed} with each such $\hat{\x}$ substituted for $\bar{\x}$ gives the result.\\
$\left(3 \rightarrow 1\right)$ Given $(\x,y)=(\hat{\x},m(\hat{\x}))$ satisfies \eqref{sum24} exactly at every extreme point $\hat{\x}$ of $X$ having $\hat{x}_j=\ell \; \forall \; j \in N_1$ and $\hat{x}_j=u\; \forall \; j \in N_2,$ the ``only if" direction of \eqref{newbeed} with each such $\hat{\x}$ substituted for $\bar{\x}$ gives us that every $\pi_K>0$ has $F_{K}(x)=0$ at every such $\hat{\x}.$ Implication $3 \rightarrow 1$ of \mythref{fite00} then gives us that every $\pi_K>0$ has $F_{K}(x)=0$ at $\tilde{\x}.$ Then the ``if" direction of \eqref{newbeed} with $\bar{\x}=\tilde{\x}$ gives the result.
\end{proof}

We invoke \mythref{foundation45} and \mythref{fitt} to establish a theorem and two corollaries. The theorem gives, in terms of the extreme points of the simplex $\simplex$ of \eqref{simpleex}, a necessary and sufficient condition for a valid core inequality \eqref{sum24} to be satisfied exactly at a point $(\tilde{\x},\tilde{y})\in G.$ This theorem is a generalization of \mythref{foundation45} in that, by restricting $n=(p+b)$ within the theorem, the stated point $(\tilde{\x},\tilde{y})\in G$ must be the extreme point $(\tilde{\x},m(\tilde{\x}))\in G$ of $\conv{G}$ found within \mythref{foundation45}. The corollaries are special cases of the theorem when the valid core inequality \eqref{sum24}: is satisfied exactly at $(\x^j,m(\x^j))$ for all $j \in \{0, \ldots, n\},$ and has all coefficients $\beta_j$ equal to the same scalar, say $\bar{\beta},$ respectively.

\begin{theorem}\thlabel{exact1111}
Given a point $(\tilde{\x},\tilde{y})\in G$ with $\tilde{\x}$ not an extreme point of the simplex $\simplex$ of \eqref{simpleex}, let $r$ be the smallest index $j$ such that $\tilde{x}_j<u$ and $s$ be the largest index $j$ such that $\tilde{x}_j> \ell.$ Further let $p$ be the number of entries of $\tilde{\x}$ having value $u$ and $b$ be the number of entries of $\tilde{\x}$ having value $\ell.$ Then $(\tilde{\x},\tilde{y})$ satisfies a valid core inequality \eqref{sum24} exactly if and only if $(\x,y)=(\x^j,m(\x^j))$ satisfies \eqref{sum24} exactly for all $j \in\{p,\ldots,n-b\},$ and $\beta_r=\ldots=\beta_s.$ 
\end{theorem}
\begin{proof}
If $n=(p+b)$ so that all entries of $\tilde{\x}$ take either value $\ell$ or $u,$ then $(\tilde{\x},\tilde{y})\in G$ is the extreme point $(\tilde{\x},m(\tilde{\x}))$ of $\conv{G},$ so the result is \mythref{foundation45}. Otherwise, $n \geq (p+b+1),$ and we adopt the notation of \mythref{fitt} that $N_1 \equiv\{j:\tilde{x}_j=\ell\},$ $N_2 \equiv\{j:\tilde{x}_j=u\},$ and $N_3 \equiv\{j:\ell<\tilde{x}_j<u\},$ so that $|N_1|=b,$ $|N_2|=p,$ and $|N_3|=n-(p+b).$ For any chosen $k \in\{0,\ldots,n-(p+b)\},$ define a new point $\hat{\x}$ in terms of $\tilde{\x}$ by setting $k$ values of $\hat{x}_j$ with $j \in N_3$ to the value $u$ and the remaining $(n-(p+b)-k)$ values of $\hat{x}_j$ with $j \in N_3$ to the value $\ell,$ and by setting $\hat{x}_j=\tilde{x}_j$ for all $j \in (N_1 \cup N_2).$ Then, if $\hat{\x} \neq \x^{k+p},$ \mythref{foundation45} gives us that $(\x,y)=(\hat{\x},m(\hat{\x}))$ satisfies \eqref{sum24} exactly if and only if $(\x^{k+p},m(\x^{k+p}))$ satisfies \eqref{sum24} exactly, and $\beta_{\hat{r}}=\ldots=\beta_{\hat{s}},$ where $\hat{r}$ is the smallest index $j$ such that $\hat{x}_j = \ell$ and $\hat{s}$ is the largest index $j$ such that $\hat{x}_j=u.$ Consequently, by considering all $k \in\{0,\ldots,n-(p+b)\}$ and all corresponding ${n-(p+b)} \choose {k}$ possible fixings of $\hat{x}_j$ for $j \in N_3,$ $(\x,y)=(\hat{\x},m(\hat{\x}))$ will satisfy \eqref{sum24} exactly at every extreme point $\hat{\x}$ of $X$ having $\hat{x}_j= \ell$ for all $j \in N_1$ and $\hat{x}_j= u$ for all $j \in N_2$ if and only if $(\x,y)=(\x^j,m(\x^j))$ satisfies \eqref{sum24} exactly for all $j \in\{p,\ldots,n-b\},$ and $\beta_r=\ldots=\beta_s.$ The ``only if" consequence that $\beta_r=\ldots=\beta_s$ is due to the following: if $r=s,$ the result is trivial while if $r<s,$ then any $\hat{\x}$ defined as above so that $\hat{x}_r=\ell$ and $\hat{x}_s=u,$ regardless of $k,$ will yield $r=\hat{r}$ and $s=\hat{s}.$ Then implication $3 \rightarrow 1$ of \mythref{fitt} gives us that $(\tilde{\x},\tilde{y})$ satisfies \eqref{sum24} exactly, as desired.
\end{proof}

\mythref{exact1111} simplifies when the valid core inequality \eqref{sum24} is satisfied exactly at $(\x^j,m(\x^j))$ for all $j \in \{0, \ldots, n\}.$ In this case, the parameters $p$ and $b$ within the theorem are no longer needed, as they serve only to restrict the points $(\x^j,m(\x^j))$ that must satisfy the inequality exactly. The simplification is below.

\begin{corollary}\thlabel{exactly1110}
Given a point $(\tilde{\x},\tilde{y})\in G$ with $\tilde{\x}$ not an extreme point of the simplex $\simplex$ of \eqref{simpleex}, let $r$ be the smallest index $j$ such that $\tilde{x}_j<u$ and $s$ be the largest index $j$ such that $\tilde{x}_j> \ell.$ Given a valid core inequality \eqref{sum24} that is satisfied exactly at $(\x^j,m(\x^j))$ for all $j \in \{0, \ldots, n\},$ the point $(\tilde{\x},\tilde{y})$ satisfies this inequality exactly if and only if $\beta_r=\ldots=\beta_s.$ 
\end{corollary}

\mythref{exact1111} also simplifies when the valid core inequality \eqref{sum24} has, for some scalar $\bar{\beta},$ $\beta_j=\bar{\beta}$ for all $j \in N.$ In this case, the parameters $r$ and $s$ within the theorem are no longer needed, as they serve only to restrict a subset of the coefficients $\beta_j$ to equal. The simplification is stated formally below.

\begin{corollary}\thlabel{exactly1111}
Given a point $(\tilde{\x},\tilde{y})\in G$ with $\tilde{\x}$ not an extreme point of the simplex $\simplex$ of \eqref{simpleex}, let $p$ be the number of entries of $\tilde{\x}$ having value $u$ and $b$ be the number of entries of $\tilde{\x}$ having value $\ell.$ Then $(\tilde{\x},\tilde{y})$ satisfies a valid core inequality \eqref{sum24} having $\beta_j=\bar{\beta}$ for all $j \in N$ exactly if and only if $(\x,y)=(\x^j,m(\x^j))$ satisfies \eqref{sum24} exactly for all $j \in\{p,\ldots,n-b\}.$ 
\end{corollary}

\section{Supermodular Functions}	\label{sec:supermod}

In this subsection, we describe $\conv{G}$ for SMPs $m(\x)$ that are \emph{supermodular} over the extreme points of $X$. Such a function $m(\x)$ over these $2^n$ points can be expressed as a set function $f(A)$ defined over $A \subseteq N$ so that $f(A)=m(\x)$ when evaluated at that extreme point $\x$ having $x_j=u$ for $j \in A$ and $x_j=\ell$ for $j \notin A.$  Recall that a set function $f$ is defined to be supermodular over $N$ if and only if
\begin{equation}
f(S \cup \{r\})-f(S) \leq f(S \cup \{r,t\})-f(S \cup \{t\}) \; \mbox{ for } r,t \in N, r \neq t, \mbox{ and }S\subseteq N \backslash \{r,t\}, \label{rest0777}
\end{equation}
we have that the function $m(\x)$ is supermodular over the $2^n$ extreme points of $X$ if and only if 
\begin{equation}
m(\x^{k})-m(\x^{k-1}) \leq m(\x^{k+1})-m(\x^{k}), \quad  \forall \; k \in \{1, \ldots, n-1\}. \label{rest0}
\end{equation}
This equivalence follows by considering, for each $k \in \{1, \ldots, n-1\},$ all sets $S \subseteq N$ within \eqref{rest0777} having $|S|=(k-1),$ and all $\{r,t\} \in N \backslash S,$ and by invoking the symmetry of $m(\x)$ to obtain that $f(T)=m(\x^k)$ for all $T \subseteq N$ with $|T|=k.$ The function $m(\x)$ is \emph{strictly supermodular} over the $2^n$ extreme points of $X$ if and only if the $(n-1)$ inequalities of \eqref{rest0} are satisfied strictly. Henceforth, for brevity, we will refer to functions $m(\x)$ that are (strictly) supermodular over the $2^n$ extreme points of $X$ as being (strictly) supermodular.

The below theorem explicitly states all core facets for supermodular $m(\x).$

\begin{theorem} \thlabel{result1}
When $m(\x)$ is supermodular, there exist at most $(n+1)$ core facets for $\conv{G},$ and all such facets, subject to repetition, are 
\begin{equation}
\frac{um(\x^0)-\ell m(\x^n)}{u-\ell}+\sum_{j=1}^n\left(\frac{m(\x^j)-m(\x^{j-1})}{u-\ell}\right)x_j-y \geq 0, \label{res1}
\end{equation}
and
\begin{equation}
-m(\x^k) - \left(\frac{m(\x^{k})-m(\x^{k-1})}{u-\ell}\right)\left(\sum_{j=1}^nx_j-[ku+(n-k)\ell]\right)+y \geq 0, \quad \forall \; k \in N. \label{res2}
\end{equation}
\end{theorem}
\begin{proof}
The $(n+1)$ inequalities of \eqref{res1} and \eqref{res2} are valid for $\conv{G}$ by \mythref{foundation}, as they are readily verified to hold for all $(\x^k,m(\x^k)), k \in\{0, \ldots, n\}.$ Inequality~\eqref{res1} is a facet because these same $(n+1)$ affinely independent points satisfy it exactly, and it is a core facet because the $\beta_j$ are nondecreasing by \eqref{rest0}. In addition, no other core facet can exist with $\beta^{\prime}=-1$ by \mythref{foundation24} because the functions ${\cal L}_k(\beta_0,\x[\beta])$ from \eqref{anoto25}, with $(\beta_0,\x[\beta])$ defined in terms of \eqref{res1}, have ${\cal L}_k(\beta_0,\x[\beta])=m(\x^k)$ for all $k \in \{0, \ldots, n\}.$ Relative to \eqref{res2}, for each $k \in N,$ the corresponding inequality is a core facet by \mythref{equalcoeff}, as it is satisfied exactly at the two points $(\x^{k-1},m(\x^{k-1}))$ and $(\x^k,m(\x^k)).$ To show that no other core facet can exist with $\beta^{\prime}=1$ and complete the proof, it is sufficient to show that every core facet $\beta_0+\sum_{j=1}^n\beta_jx_j+y \geq 0$ which is satisfied exactly at some $(\x^{r},m(\x^{r}))$ and $(\x^{s},m(\x^{s}))$ with $r<s,$ is also satisfied exactly at $(\x^{r+1},m(\x^{r+1})).$ Then induction has the core facet being satisfied exactly at $(\x^{p},m(\x^{p}))$ for all $p \in \{r, \ldots, s\}$ so that, for each $p \in \{r+1, \ldots, s\},$ by inserting $(\x^{p-1},m(\x^{p-1}))$ and $(\x^{p},m(\x^{p}))$ into the facet and subtracting the first expression from the second, we obtain $\beta_{p}(u-\ell)=-\left(m(\x^{p})-m(\x^{p-1})\right).$ As the $\beta_p$ are nondecreasing and the values $-\left(m(\x^{p})-m(\x^{p-1})\right)$ are nonincreasing, it will then follow that $\beta_{p}(u-\ell)=-\left(m(\x^{r+1})-m(\x^r)\right)$ for all $p \in \{r+1, \ldots, s\}.$ Hence, by selecting $r=\mbox{argmin}_{k \in \{0, \ldots, n-1\}}\left\{{\cal L}_k(\beta_0,\x[\beta]) +m(\x^k)=0\right\}$ and $s=\mbox{argmax}_{k \in \{1, \ldots, n\}}\left\{{\cal L}_k(\beta_0,\x[\beta]) + m(\x^k)=0\right\}$, \mythref{equalcoeff8} will give us that $\beta_j=-\left(\frac{m(\x^{r+1})-m(\x^r)}{u-\ell}\right)$ for all $j \in N.$ Then the chosen core facet must be inequality \eqref{res2} with $k=(r+1).$ 

To show that the core facet is satisfied exactly at $(\x^{r+1},m(\x^{r+1})),$ we restrict attention to $s \geq (r+2)$ since the result is trivial for $s=(r+1).$ We have that

\begin{eqnarray*}
m(\x^{r+1}) &\le& m(\x^{r})+\frac{m(\x^s)-m(\x^r)}{s-r} \\
&=&-\beta_0-u\left(\sum_{j=1}^{r}\beta_j\right)-\ell\left(\sum_{j=r+1}^{n}\beta_j\right)-\left(\frac{u-\ell}{s-r}\right)\left(\sum_{j=r+1}^{s}\beta_j\right) \\
&=&-\beta_0-u\left(\sum_{j=1}^{r+1}\beta_j\right)-\ell\left(\sum_{j=r+2}^{n}\beta_j\right)-\left(\frac{u-\ell}{s-r}\right)\left(\sum_{j=r+1}^{s}(\beta_j-\beta_{r+1})\right) \\
&\leq& -\beta_0-u\left(\sum_{j=1}^{r+1}\beta_j\right)-\ell\left(\sum_{j=r+2}^{n}\beta_j\right) \\
&\leq& m(\x^{r+1}),
\end{eqnarray*}
where the first inequality follows from summing the $(s-r)$ inequalities $m(\x^{r+1})-m(\x^r) \leq m(\x^{k+1})-m(\x^k)$ for $k \in \{r, \ldots, s-1\}$ that are implied by \eqref{rest0} for $k>r,$ the first equality follows from the facet holding exactly at $(\x^r,m(\x^r))$ and $(\x^{s},m(\x^{s})),$ the second equality is algebra, the second inequality follows from the nondecreasing values of $\beta_j,$ and the final inequality follows from the feasibility of $(\x^{r+1},m(\x^{r+1}))$ to $\conv{G}.$ Then $m(\x^{r+1})=-\beta_0-u\left(\sum_{j=1}^{r+1}\beta_j\right)-\ell\left(\sum_{j=r+2}^{n}\beta_j\right)$ so that the core facet is satisfied exactly at $(\x^{r+1},m(\x^{r+1})).$ The proof is complete.
\end{proof}

\mythref{result1} states that there exist \emph{at most} $(n+1)$ core facets because the inequalities of \eqref{res2} can repeat. Repetition will occur whenever $\left(m(\x^{p})-m(\x^{p-1})\right)=\left(m(\x^{q})-m(\x^{q-1})\right)$ for distinct $p,q \in N.$ Notably, if $m(\x)$ is strictly supermodular, then  repetitions will not occur so that \eqref{res2} will contain $n$ distinct facets. 

Subject to permutations of $\x[\beta],$ inequalities \eqref{res1} define the concave envelope of $m(\x).$ \akshay{These inequalities are a special case of the polymatroid inequalities that are known for general supermodular functions \cite{lovasz1983submod,tawarmalani2013explicit}. Polymatroid inequalities are known to be separated easily in $O(n\log{n})$ time using a sorting algorithm since \citet{edmonds1970submodular} showed how to optimize a linear function over the polymatroid polyhedron corresponding to a submodular function.} Inequalities~\eqref{res2} have equal coefficients on the variables and hence they do not need to be permuted and define the convex envelope of $m(\x)$. \akshay{Our proof for these inequalities is an alternative proof to that of \citet[Theorem 4.6]{tawarmalani2013explicit}.} For a submodular function $m(\x)$ (meaning that the inequality in \eqref{rest0} is switched to $\ge$), inequalities~\eqref{res1}, subject to permutations of $\x[\beta],$ define the convex envelope, and inequalities~\eqref{res2} define the concave envelope subject to the following modifications:  all occurrences of $m(\x^j),$ $j \in \{0, \ldots, n\},$ and $y$ are negated.

Now we use \mythref{exactly1110,exactly1111} to identify, for each of the $(n+1)$ core facets \eqref{res1} and \eqref{res2} of \mythref{result1}, the set of all points $(\x,y) \in G$ that satisfies it exactly. A key ingredient  of \mythref{result1} that is useful in our upcoming proof is that every core facet $\beta_0+\sum_{j=1}^n\beta_jx_j+y \geq 0$ of the form \eqref{res2} which is satisfied exactly at some $(\x^{r},m(\x^{r}))$ and $(\x^{s},m(\x^{s}))$ with $r<s,$ is also satisfied exactly at $(\x^{p},m(\x^{p}))$ for all $p \in \{r, \ldots, s\}.$ We then have that the set of points $(\x^j,m(\x^j))$ which satisfies the facet exactly must be consecutive. 

\begin{theorem}	\thlabel{newsuper100}
Let $m(\x)$ be supermodular and $(\x,y)\in G$.
\begin{enumerate}
\item $(\x,y)$ satisfies \eqref{res1} exactly if and only if $\x=\x^{k}$ for some $k$, or $\beta_{r}=\cdots=\beta_{s}$ where $r=\min\{j\colon x_{j}<u\}$ and $s = \max\{j\colon x_{j} > \ell\}$.
\item $(\x,y)$ satisfies \eqref{res2} exactly if and only if $(\x^{j},m(\x^{j}))$ satisfies \eqref{res2} exactly for $j=p,\dots,n-q$, where $p=|\{j\colon x_{j}=u \} |$ and $q=|\{j\colon x_{j}=\ell \} |$.
\item $(\x,y)$ satisfies \eqref{res2} exactly if and only if $|\{j\colon x_{j}=u \}| \ge v$ and $|\{j\colon x_{j}=\ell \}| \ge n-w$, where $v$ and $w$ are the smallest and largest, respectively, values of $j$ such that the point $(\x^{j},m(\x^{j}))$ satisfies \eqref{res2} exactly.
\end{enumerate}
\end{theorem}
\begin{proof}
(1) As noted in the proof of \mythref{result1}, the points $(\x^k,m(\x^k))$ satisfy \eqref{res1} exactly for all $k \in \{0, \ldots, n\}.$ Then the result trivially holds true when ${\x}$ is an extreme point of $\simplex,$ and it holds true when ${\x}$ is not an extreme point of $\simplex$ by \mythref{exactly1110}.

(2) The result trivially holds true when ${\x}$ is an extreme point of $\simplex,$ and it holds true when ${\x}$ is not an extreme point of $\simplex$ by \mythref{exactly1111} with $\bar{\beta}=-\left(\frac{m(\x^k)-m(\x^{k-1}}{u-\ell}\right).$

(3) This is a restatement of above using Remark 6 of \mythref{result1} which states that the set of extreme points of $\simplex$ satisfying the facet exactly must be consecutive.
\end{proof}

More refined characterizations for exactness hold when the function is strictly supermodular.

\begin{corollary} \thlabel{exactstrictsupermod}
Let $m(\x)$ be strictly supermodular and $(\x,y)\in G$.
\begin{enumerate}
\item $(\x,y)$ satisfies \eqref{res1} exactly if and only if $\x$ is a convex combination of two consecutive extreme points $\x^j$ and $\x^{j+1}$ of $\simplex.$
\item $(\x,y)$ satisfies \eqref{res2} exactly if and only if $|\{j\colon x_{j}=u_{j} \} | \ge k-1$ and $| \{j\colon x_{j} = \ell \}| \ge n-k$.
\end{enumerate} 
\end{corollary}
\begin{proof}
(1) Follows directly from \mythref{newsuper100} since the coefficients $\beta_j= \left(\frac{m(\x^j)-m(\x^{j-1})}{u-\ell}\right)$ for all $j \in N$ of \eqref{res1} are distinct, as every inequality within \eqref{rest0} is satisfied strictly for strictly supermodular $m(\x).$

(2) Consider any $k \in N.$ As noted in the proof of \mythref{result1} and readily verified, $(\x,y)=(\x^{k-1},m(\x^{k-1}))$ and $(\x,y)=(\x^k,m(\x^k))$ each satisfy the core facet exactly. Thus, it is sufficient to show that supermodular $m(\x)$ enforces $r=(k-1)$ and $s=k$ in \mythref{newsuper100}. By contradiction, suppose that $s \geq (r+2)$ so that each of $(\x^{r},m(\x^{r})),$ $(\x^{r+1},m(\x^{r+1})),$ and $(\x^{r+2},m(\x^{r+2}))$ satisfy the facet exactly. As in the proof of \mythref{result1}, for each $p \in \{r+1,r+2\},$ by inserting $(\x^{p-1},m(\x^{p-1}))$ and $(\x^p,m(\x^p))$ into the facet and subtracting the second expression from the first, we obtain
\begin{equation}
 \frac{-\left(m(\x^{r+1})-m(\x^r)\right)}{u-\ell}=\beta_{r+1}=\beta_{r+2}=\frac{-\left(m(\x^{r+2})-m(\x^{r+1})\right)}{u-\ell},\nonumber
\end{equation}
where $\beta_{r+1}=\beta_{r+2}$ is due to the form of \eqref{res2}. The strict inequality of \eqref{rest0} for supermodular $m(\x)$ with $k=(r+1)$ yields the contradiction that $\beta_{r+1} > \beta_{r+2}.$
\end{proof}

We finish this section by remarking on supermodularity of an SMP, giving us some important families of functions $m(\x)$ so that \mythref{result1} characterizes their convex hull. First we provide an alternate characterization to~\eqref{rest0} by expressing these requirements in terms of the coefficients $c_{j}$ of the function $m$.

\begin{proposition}	\thlabel{supermodcond}
$m(\x)$ is supermodular over $X$ if and only if $\sum_{d=2}^{n}c_{d}(u-\ell)^{2}\vartheta_{k,d} \ge 0$ for $k=1,\dots,n-1$, where $\vartheta_{2,2} = 1$ and
\[
\vartheta_{k,d} = \sum_{\substack{J \subseteq N\setminus\{k,k+1\} \\ |J|=d-2}}\,\prod_{j \in J}x^{k-1}_j, \qquad d = 3,\dots,n,\ k =1,\dots,n-1,
\]
\end{proposition}
\begin{proof}
For each $k \in \{1, \ldots, n-1\},$ the difference $m(\x^{k+1})-m(\x^{k})-\left(m(\x^{k})-m(\x^{k-1})\right)$ is computable in terms of only those expressions $\prod_{j \in J}x_j$ within $m(\x)$ that contain both $k \in J$ and $\{k+1\} \in J.$ For $d=2,$ this difference for the $n \choose d$ expressions in $m(\x)$ of degree $d$ is given by $c_2(u-\ell)^2,$ while for each $d \in \{3, \ldots, n\},$ it is given by
\begin{equation} 
c_d(u-\ell)^2\left(\sum_{\substack{J \subseteq N-\{k,k+1\} \\ |J|=d-2}}\left(\prod_{j \in J}x^{k-1}_j\right)\right). \label{newbee}
\end{equation}
Then \eqref{rest0} is equivalent to having the sum of these expressions from 2 to $n$ being nonnegative for all $k \in \{1, \ldots, n-1\}.$ 
\end{proof}

\akshay{
A multilinear monomial $x_{1}x_{2}\cdots x_{n}$ is supermodular over the nonnegative orthant \citep[cf.][]{tawarmalani2013explicit}. Since supermodularity is preserved under taking nonnegative combinations of functions, it follows that an SMP with $c_{j}\ge 0$ for all $j$ is supermodular over $\real^{n}_{+}$. This fact for nonnegative valued SMPs also follows from our characterization.}

\begin{corollary}	\thlabel{nonnegsupermod}
$m(\x)$ is supermodular over $X$ if $\ell \ge 0$ and $c_{j} \ge 0$ for all $j=2,\dots,n$.
\end{corollary}
\begin{proof}
$\ell\ge 0$ implies that $\x^{k}_{j} \ge 0$ for all $j,k$, which implies $\vartheta_{k,d}\ge 0$ for all $k,d$. The summation in \mythref{supermodcond} is nonnegative because $c_{j}\ge 0$ for all $j$.
\end{proof}

Another consequence is that symmetric quadratic polynomials are always either submodular or supermodular, regardless of the box in $\real^{n}$.

\begin{corollary}
The symmetric quadratic polynomial $\tau\sum_{i\neq j}x_{i}x_{j}$ is either submodular or supermodular over $X$ for any $\tau\neq 0$.
\end{corollary}
\begin{proof}
A symmetric quadratic polynomial is $m(\x)$ with $c_{3} = \cdots = c_{n} = 0$, and $c_{2}=\tau$. Since all monomials are of the same degree 2, we can assume wlog that $\ell\ge 0$ because otherwise we can negate all the variables and consider the reflected box $X' = \{x'\sep -u_{j}\le x'_{j} \le -\ell_{j} \}$. Applying \mythref{nonnegsupermod} to $m(\x)$ if $\tau > 0$ or to $-m(\x)$ if $\tau < 0$ yields the desired claim. 
\end{proof}

When considering the unit hypercube in $\real^{n}$, supermodularity is attained through nonnegativity of a partial sum. Denote ${0 \choose 0}=1$.

\begin{corollary}
$m(\x)$ is supermodular over $[0,1]^{n}$ if and only if
\begin{equation}
\sum_{d=2}^{k+1}{{k-1} \choose {d-2}}c_d \geq 0, \quad \forall \; k \in \{1, \ldots, n-1\}. \label{newbee2}
\end{equation} 
\end{corollary}
\begin{proof}
Follows immediately by simplifying the sum in~\eqref{newbee}.
\end{proof}


More general families of SMPs are encompassed by \eqref{newbee2} including, for example, those having $c_d \geq 0$ for $d\in\{2, \ldots, n-1\}$ and $c_n \geq -\sum_{d=2}^{n-1}{{n-2} \choose {d-2}}c_d,$ which reduces to $c_2 \geq \mbox{max}\{0,(2-n)c_3\}$ for cubic functions.

\section{Monomials with Reflection Symmetry}

\akshay{We consider $m(\x)$ to be a monomial $m(\x)=c_n\prod_{j=1}^n x_j$ over the box $X$. The coefficient $c_{n}$ can be taken to be 1 since we can scale $y$ to $y/c_{n}$. Three cases arise depending on the location of the box in $\real^{n}$: (1) $\ell u = 0$, (2) $\ell u > 0$, and (3) $\ell u < 0$. After performing appropriate scalings, the first two cases are equivalent to that of $\ell = 0, u =1$ and $\ell = 1, u = r$ for some fixed $r > 1$, respectively. These two cases fall within the class of supermodular SMPs (cf.~\mythref{nonnegsupermod}), and so \mythref{result1} gives their convex hull. Indeed, it is easily verified that the resulting description of the convex hull matches the known envelopes for $\prod_{j=1}^{n} x_{j}$ over $[0,1]^{n}$ \cite{crama1993concave} and over $[1,r]^{n}$ for some fixed $r > 1$ (see \cite[Proposition 4.1]{monomerr} which is a direct consequence of results from \cite[Theorem 1]{benson2004concave} and \cite[Theorem 4.6]{tawarmalani2013explicit}). The third case $\ell u < 0$ is equivalent (upto scaling) to taking $\ell = -1$ and $u = r$ for some $r > 0$ so that $X = [-1,r]^{n}$, and it is readily verified that a monomial is neither submodular nor supermodular over $[-1,r]^{n}$. Note that since we are considering monomials, scaling means that if each variable $x_{j}$ has different lower and upper bounds $\ell_{j}$ and $u_{j}$, then as long as $u_{j} = -\ell_{j}$ we can reduce to the case $X = [-1,r]^{n}$. The convex hull for the specific subcase having $r=1$ was first established by the authors in \citep[Theorem 4.1]{monomerr}.  However, it was done so without using the symmetry of the monomial and hence did not recognize the core facets. Our main goal in this section is to independently describe this convex hull by identifying the core facets and their properties. 
}


\begin{theorem}	\thlabel{resulting3}
For $m(\x)=\prod_{j=1}^nx_j$ and $-\ell=u=1,$ there exist precisely $n+3$ core facets, and these are 
\begin{align}
&1-y \geq 0 \mbox{ and } 1+y \geq 0 \label{trivial} \\
&(n-1)+\sum_{j=1}^nx_j+(-1)^ny \geq 0, \label{pair1} \\
&(n-1)-\sum_{j=1}^nx_j+y \geq 0. \label{pair2}\\
&(n-1)-\left(\sum_{j=1}^{t+1}x_j-\sum_{j=t+2}^nx_j\right) +(-1)^{n-(t+1)}y \geq 0, \quad \forall \; t \in \{0, \ldots, n-2\}.  \label{dominate2}
\end{align}
\end{theorem}

\akshay{The general case of convexifying $\prod_{j=1}^{n}x_{j}$ over $[-1,r]^{n}$ for $r > 0, r \neq 1$ remains an open question. Before giving our proof for the above theorem, let us note the complexity of separation and comment on the structure of the proposed inequalities which includes identifying exactness of the core facets. } 

\begin{corollary}
For $m(\x)=\prod_{j=1}^nx_j$ and $-\ell=u=1,$ a point can be separated from $\conv{G}$ in $O(n\log{n})$ time.
\end{corollary}
\begin{proof}
\mythref{resulting3} tells us that there are $t=n+3$ core facets. \mythref{sepcompl} then implies that the complexity of separation is $O(n^{2} + n\log{n})$. In the proof of this result, the complexity $O(nt)$ came from having to take a summation of $n$ terms while separating each of the $t$ core facets. Since the summation in both \eqref{pair1} and \eqref{pair2} is the same for any permutation of variables and therefore can be stored and reused, we get that the overall complexity is $O(t+n\log{n}) = O(n\log{n})$.
\end{proof}

\subsection{Structure of the inequalities}

All facets for $\conv{G}$ are computable in terms of the core facets enumerated in \mythref{resulting3} via permutations of $\x[\beta].$ Since each of the four core facets of \eqref{trivial}, \eqref{pair1}, and \eqref{pair2} have all $\beta_j$ equal, no permutations of $\x[\beta]$ will produce additional facets. However, for each $t \in \{0, \ldots, n-2\},$ the core facet \eqref{dominate2} admits ${n} \choose {t+1}$ facets. Then \eqref{pair1}, \eqref{pair2}, and \eqref{dominate2} motivate the $2^n$ inequalities
\begin{equation}
(n-1)-\left(\sum_{j \in J}x_j-\sum_{j \in N\setminus J}x_j\right)+(-1)^{n-|J|}y \geq 0, \quad \forall \; J \subseteq N,  \label{dominate3}
\end{equation}
where \eqref{dominate3} with $J = \emptyset$ is \eqref{pair1}, where \eqref{dominate3} with $J=N$ is \eqref{pair2}, and where \eqref{dominate3} for each $J \subseteq N$ with $|J| \in \{1, \ldots, n-1\}$ are the ${n \choose |J|}$ facets obtained by permutations of $\x[\beta]$ in \eqref{dominate2} when $(t+1)=|J|.$ The desired representation of $\conv{G}$ is then the two inequalities $-1 \leq y \leq 1$ of \eqref{trivial}, the $2^n$ inequalities of \eqref{dominate3} and, for the case in which $n \geq 3,$ the $2n$ inequalities $-1 \leq x_j \leq 1 \; \forall \; j \in N.$ We can combine and more succinctly write the $(2+2^n+2n)$ inequalities of \eqref{trivial}, \eqref{dominate3}, and the restrictions $-1 \leq x_j \leq 1$ for $j \in N$ for the cases having $n \geq 3,$ by letting the variable $x_{n+1}$ denote $y$ and by letting $N^{\prime} = \{1, \ldots, n+1\}.$ Using this new definition of variables, \eqref{dominate3} can be partitioned in terms of the exponents $(n-|J|)$ as
\begin{equation}
(n-1)-\left(\sum_{j \in J}x_j-\sum_{j \in (N\setminus J) \cup \{n+1\}}x_j\right) \geq 0,\quad \forall \; J \subseteq N \mbox{ with }n-|J| \mbox{ even},  \nonumber
\end{equation}
and
\begin{equation}
(n-1)-\left(\sum_{j \in J\cup \{n+1\}}x_j-\sum_{j \in N\setminus J}x_j\right) \geq 0,\quad \forall \; J \subseteq N \mbox{ with }n-|J| \mbox{ odd}.  \nonumber
\end{equation}
Consequently, $-1 \leq x_j \leq 1$ for $j \in N,$ \eqref{trivial}, and \eqref{dominate3} can be expressed as \eqref{con258} and \eqref{con269} below, where \eqref{con258} encompasses $-1 \leq x_j \leq 1$ for $j \in N$ and \eqref{trivial}, and \eqref{con269} encompasses the above two families of inequalities.
\begin{eqnarray}
-1 \leq x_j \leq 1 \; &\forall \; j \in N^{\prime} \label{con258}\\
(n-1)-\left(\sum_{j \in J}x_j-\sum_{j \in N^{\prime}\setminus J}x_j\right) \geq 0, &\quad \forall \;J \subseteq N^{\prime} \mbox{ with } n+1-|J| \mbox{ odd} \label{con269}
\end{eqnarray}
Of course, if $n=2,$ then \eqref{con258} simplifies to $-1 \leq x_{n+1} \leq 1.$

Now we use \mythref{exact1111} and \mythref{exactly1111} to identify, for each of the $(n+3)$ core facets \eqref{trivial}, \eqref{pair1}, \eqref{pair2}, and \eqref{dominate2} of \mythref{resulting3}, the set of all points $(\x,y) \in G$ that satisfies it exactly. 

\begin{theorem}\thlabel{finalet}
For $m(\x)=\prod_{j=1}^nx_j$ and $-\ell=u=1,$ $(\tilde{\x},\tilde{y}) \in G$ satisfies the core facet
\begin{enumerate}
	\item \eqref{trivial} exactly if and only if $\tilde{\x}$ contains
	\begin{enumerate}
	\item even number of entries of value $-1$ and the remaining entries of value 1 when $\beta^{\prime}=-1,$ 
	\item odd number of entries of value $-1$ and the remaining entries of value 1 when $\beta^{\prime}=1,$  
	\end{enumerate}
	\item \eqref{pair1} exactly if and only if $\tilde{\x}$ contains at least $(n-1)$ entries of value $-1,$ 
	\item \eqref{pair2} exactly if and only if $\tilde{\x}$ contains at least $(n-1)$ entries of value $1,$
	\item \eqref{dominate2} for $t \in \{0, \ldots, n-2\}$ if and only if $\tilde{\x}$ differs from $\x^{t+1}$ in at most one entry.
\end{enumerate}
\end{theorem}
\begin{proof}
Each statement is considered separately. 
\begin{enumerate}
\item Follows directly from the definition of $m(\x)$ when $-\ell =u=1.$ 
\item \mythref{foundation900} and its proof give us that \eqref{pair1} is satisfied exactly at $(\tilde{\x},\tilde{y})=(\x^{k},m(\x^{k}))$ for $k \in \{0,1\},$ but at no other $k.$ Then \mythref{exactly1111} with $\bar{\beta}=1$ and $\tilde{\x}$ having $p=0$ and $b=(n-1)$ gives the result.
\item \mythref{foundation900} and its proof give us that \eqref{pair2} is satisfied exactly at $(\tilde{\x},\tilde{y})=(\x^{k},m(\x^{k}))$ for $k \in \{n-1,n\},$ but at no other $k.$ Then \mythref{exactly1111} with $\bar{\beta}=-1$ and $\tilde{\x}$ having $p=(n-1)$ and $b=0$ gives the result.
\item For each $t \in \{0, \ldots, n-2\},$ \mythref{resulting} and its proof give us that the associated inequality of \eqref{dominate2} is satisfied exactly at $(\tilde{\x},\tilde{y})=(\x^{k},m(\x^{k}))$ for $k \in \{t,t+1,t+2\},$ but at no other $k.$ Then \mythref{exact1111} with $\tilde{\x}$ having $p=t$ and $b=(n-t-1)$ allows $\tilde{\x},$ with suitable $r$ and $s,$ to differ from $x^{t+1}$ in only a single entry $j$ less than or equal to $(t+1).$ Similarly, \mythref{exact1111} with $\tilde{\x}$ having $p=(t+1)$ and $b=(n-t-1)$ allows $\tilde{\x},$ with suitable $r$ and $s,$ to differ from $x^{t+1}$ in only a single entry $j$ greater than or equal to $(t+1).$ Combining these outcomes gives the result.\qedhere
\end{enumerate}
\end{proof}

\subsection{Proof of the Convex Hull}

We prove \mythref{resulting3} using three main lemmas. Let us begin with two simple observations. First, there exist no facets \eqref{sum24} with $\beta^{\prime}=0$ if $n=2,$ and that there exist $2n$ such facets of the form $-1 \leq x_j \leq 1$ for $j \in N$ if $n \geq 3.$ Second, the two inequalities of \eqref{trivial} are core facets. They are trivially valid, and they are core facets by \mythref{equalcoeff} because each has $\beta_j=0$ for all $j \in N,$ and each is satisfied exactly at two extreme points $\x^k$ of $\simplex,$ with one point not being $\x^0$ or $\x^n;$ equality occurs at $\x^{n-2}$ and $\x^n$ for the left inequality and occurs at $\x^{n-3}$ and $\x^{n-1}$ for the right inequality. 

Now, to facilitate the derivation of the remaining core facets and thereby characterize $\conv{G},$ we define, for each $(p,q),p<q,$ the value 
\begin{equation}
D(p,q)=2\sum_{j=p+1}^q\beta_j+\beta^{\prime}\left(m(\x^q)-m(\x^p)\right)
\nonumber 
\end{equation} 
that is computed by subtracting the left side of \eqref{sum24} evaluated at $(\x,y)=(\x^{p},m(\x^{p}))$ from the left side evaluated at $(\x^{q},m(\x^{q})),$ upon recalling that $-\ell=u=1.$

The following result identifies two core facets in terms of the extreme points $\x^k$ of $\simplex,$ and also eliminates, in terms of these same points, other inequalities from consideration. 

\begin{lemma}	\thlabel{foundation900}
The only core facets that are satisfied exactly at $(\x^k,\mk)$ for two consecutive extreme points $\x^k$ of the simplex $\simplex,$ and for no other such extreme points, are~\eqref{pair1} and \eqref{pair2}.
\end{lemma} 
\begin{proof}
Suppose that a core facet is satisfied exactly at $(\x^t,m(\x^t))$ and $(\x^{t+1},m(\x^{t+1}))$ for extreme points $\x^t$ and $\x^{t+1}$ of the simplex $\simplex,$ and for no other such extreme points. Then \mythref{equalcoeff8} with $r=t$ and $s=(t+1)$ states that there exists a $\bar{\beta}$ so that $\bar{\beta}=\beta_j$ for all $j \in N.$ Three cases exist.

\begin{itemize}
\item $t \in\{1, \ldots, n-2\}.$ We have $D(t-1,t)=2\bar{\beta}+\beta^{\prime}\left(m(\x^{t+2})-m(\x^{t+1})\right) < 0$ because $m(\x^{t})=m(\x^{t+2})$ and $m(\x^{t-1})=m(\x^{t+1}),$ and we have $D(t+1,t+2)=2\bar{\beta}+\beta^{\prime}\left(m(\x^{t+2})-m(\x^{t+1})\right) > 0.$ Combining, $0 < 2\bar{\beta}+\beta^{\prime}\left(m(\x^{t+2})-m(\x^{t+1})\right)< 0,$ which is not possible.
\item
$t=0.$ We have $D(0,1)=2\bar{\beta}+\beta^{\prime}\left(2(-1)^{n+1}\right)=0,$ with the first equality holding since $\mk[0]=-m(\x^{1})=(-1)^{n}.$ Then $\bar{\beta}=\beta^{\prime}(-1)^n$ and \eqref{sum24} evaluated at $(\x,y)=(\x^0,\mk[0])$ gives $\beta_0=\beta^{\prime}\left((-1)^n(n-1)\right)$ because $\mk[0]=(-1)^{n}.$ Inequality \eqref{sum24} becomes 
\begin{equation}
\beta^{\prime}\left((-1)^n(n-1)+(-1)^n\sum_{j=1}^nx_j+y\right) \geq 0. \nonumber
\end{equation} 
This inequality is valid when $\beta^{\prime} = (-1)^n$ by the ``if" direction of \mythref{foundation} since it holds true at $(\x,y)=(\x^k,\mk)$ for $k \in \{0, \ldots, n\},$ but it is invalid when $\beta^{\prime}=(-1)^{n+1}$ since it is then violated at $(\x,y)=(\x^k,\mk)$ for $k \in \{2, \ldots, n\}.$ The inequality with $\beta^{\prime} = (-1)^n$ is \eqref{pair1}, and is a facet by the ``if" direction of \mythref{equalcoeff} since it holds exactly at $(\x,y)=(\x^k, \mk)$ for $k \in \{0,1\}.$
\item
$t=(n-1).$ We have $D(n-1,n)=2\bar{\beta}+2\beta^{\prime}=0,$ with the first equality holding since $-m(\x^{n-1})=\mk[n]=1.$  Then $\bar{\beta}=-\beta^{\prime}$ and \eqref{sum24} evaluated at $(\x,y)=(\x^{n-1},m(\x^{n-1}))$ gives $\beta_0=\beta^{\prime}(n-1)$ because $m(\x^{n-1})=-1.$ Inequality \eqref{sum24} becomes 
\begin{equation}
\beta^{\prime}\left((n-1)-\sum_{j=1}^nx_j+y\right) \geq 0. \nonumber
\end{equation} 
This inequality is valid when $\beta^{\prime}=1$ by the ``if" direction of \mythref{foundation} since it holds true at $(\x,y)=(\x^k,\mk)$ for $k \in \{0, \ldots, n\},$ but it is invalid when $\beta^{\prime}=-1$ since it is then violated at $(\x,y)=(\x^k,\mk)$ for $k \in \{0, \ldots, n-2\}.$ The inequality with $\beta^{\prime} = 1$ is \eqref{pair2}, and is a facet by the ``if" direction of \mythref{equalcoeff} since it holds exactly at $(\x,y)=(\x^{k}, \mk)$ for $k \in \{n-1,n\}.$ \qedhere
\end{itemize}
\end{proof} 

The below result builds further by providing, in terms of the extreme points $\x^k$ of $\simplex,$ necessary conditions for valid core inequalities that are not found in \eqref{trivial}, \eqref{pair1}, or \eqref{pair2} to be facets. 

\begin{lemma}	\thlabel{foundation15}
Every core facet for $\conv{G}$ which is not found in \eqref{trivial}, \eqref{pair1}, or \eqref{pair2} is satisfied exactly at $(\x^k,\mk)$ for three consecutive extreme points $\x^k$ of the simplex $\simplex,$ and for no other such extreme points.
\end{lemma}
\begin{proof}
\mythref{foundation22} gives us that a core facet is satisfied exactly at $(\x^k,\mk)$ for at least two extreme points of the simplex $\simplex.$ Let $r$ and $s$ be 
as in \mythref{equalcoeff8}. Let $d=(s-r) \geq 1$ denote the difference between $s$ and $r.$ \mythref{foundation900} exhausts the cases having $d=1,$ and so it is sufficient to show three results: a core facet having $d=2$ that is not of the form \eqref{trivial} must be satisfied exactly at $(\x,y)=(\x^{r+1},\mk[r+1]),$ there exists no valid inequality having $d \geq 3$ and $d$ odd, and a core facet having $d \geq 4$ and $d$ even must be of the form \eqref{trivial}. 
\begin{itemize}
\item Suppose that a core facet $\beta_0+\sum_{j=1}^n\beta_jx_j +\beta^{\prime}y \geq 0$ has $d=2$ and is not of the form \eqref{trivial}. We have $\beta_1 = \ldots = \beta_{r+1}$ and $\beta_{r+2}=\ldots=\beta_n$ by \mythref{equalcoeff8}, so that we cannot have $\beta_{r+1}=\beta_{r+2}=0$ since the facet would be of the form \eqref{trivial}. As a result, $D(r,r+2)=0$ with $\mk[r]=m(\x^{r+2})$ gives $0<-\beta_{r+1} = \beta_{r+2}=\bar{\beta}$ for some $\bar{\beta}>0.$ Then $D(r+1,r+2) \leq 0$ gives $\bar{\beta} \leq \beta^{\prime}\mk[r+1]$ because $m(\x^{r+2}) = -\mk[r+1]$ and, since $0 < \bar{\beta}$ and $\mk[r+1]=(-1)^{n-(r+1)},$ we have $0 < \bar{\beta}\leq \beta^{\prime}\mk[r+1]=\beta^{\prime}(-1)^{n-(r+1)}$ so that $\beta^{\prime}=(-1)^{n-(r+1)}.$ Consequently, the core facet takes the form 
\begin{equation} 
\beta_0-\bar{\beta}\left(\sum_{j=1}^{r+1}x_j-\sum_{j=r+2}^{n}x_j\right)+(-1)^{n-(r+1)}y \geq 0. \label{trivia0}
\end{equation} 
We now use \mythref{equalcoeff7} to show that the core facet must be satisfied exactly at $(\x,y)=(\x^{r+1},\mk[r+1])$ by setting $\beta_0 = (n-1)$ and $\bar{\beta}=1$ within \eqref{trivia0}. The resulting inequality is valid for $\conv{G}$ and is satisfied exactly at $(\x,y)=(\x^k,\mk)$ for $k \in \{r,r+1,r+2\}.$ The validity follows from \mythref{foundation} since, for any $r \in \{0, \ldots, n-2\},$ the left side is $2|k-(r+1)|-1+(-1)^{k-(r+1)}$ at $(\x,y)=\left(\x^k,\mk\right)$ for each $k \in \{0, \ldots, n\}$ because $\mk=(-1)^{n-k}.$ 

\item By contradiction, suppose there exists a valid inequality $\beta_0+\sum_{j=1}^n\beta_jx_j +\beta^{\prime}y \geq 0$ having $d \geq 3$ and $d$ odd. Then $D(r,s)=0$ gives $\sum_{j=r+1}^s\beta_j=\beta^{\prime}\mk[r]$ since $\mk[r]=-\mk[s].$ If $\beta^{\prime}\mk[r]=1$ so that $\beta^{\prime}\left(\mk[r+1]-\mk[r]\right)=-2,$ then $D(r,r+1) \geq 0$ gives $\beta_{r+1} \geq 1,$ and nondecreasing $\beta_j$ gives $\beta_{r+1} \leq \frac{1}{d}.$ Thus, a contradiction.  Similarly, if $\beta^{\prime}\mk[r]=-1$ so that $\beta^{\prime}\left(\mk[s]-m(\x^{s-1})\right)=2,$ then $D(s-1,s) \leq 0$ gives $\beta_{s} \leq -1,$ and nondecreasing $\beta_j$ gives $\beta_{s} \geq \frac{-1}{d}.$ Again a contradiction.
\item Suppose that a core facet $\beta_0+\sum_{j=1}^n\beta_jx_j +\beta^{\prime}y \geq 0$ has $d \geq 4$ and $d$ even. The value $D(r,s)=0$ gives $\sum_{j=r+1}^s \beta_j=0$ since $\mk[r]=\mk[s].$ But $D(r,r+2) \geq 0$ gives $\left(\beta_{r+1}+\beta_{r+2}\right) \geq 0$ since $\mk[r]=m(\x^{r+2}),$ and the nondecreasing $\beta_j$ gives $\beta_j=0$ for $j \in \{r+1, \ldots, s\}.$ Then $\beta_j=0$ for $j \in N$ by \mythref{equalcoeff8}, and the facet is of the form \eqref{trivial}. \qedhere
\end{itemize}
\end{proof}

A key component of the proof of \mythref{foundation15} is the family of $(n-1)$ valid core inequalities \eqref{trivia0} with $\beta_0 = (n-1)$ and $\bar{\beta}=1$ that has the property that, for each $r \in \{0, \ldots, n-2\},$ the corresponding inequality is satisfied exactly at $(\x,y)=(\x^k,\mk)$ for $k \in \{r,r+1,r+2\}.$ It turns out that each such inequality is a core facet, and that there exist no other core facets that are satisfied exactly at $(\x^k,\mk)$ for three consecutive extreme points $\x^k$ of the simplex $\simplex.$ This result is established below. Here, we find it convenient to substitute the index $t$ in \eqref{dominate2} for $r$ in \eqref{trivia0}.

\begin{lemma}	\thlabel{resulting}
There exist precisely $(n-1)$ core facets that are satisfied exactly at $(\x^k,\mk)$ for three consecutive extreme points $\x^k$ of the simplex $\simplex,$ and these inequalities are~\eqref{dominate2}.
\end{lemma}
\begin{proof}
The proof consists of two parts: the first part shows that inequalities \eqref{dominate2} are the only candidate core facets that are satisfied exactly at three consecutive extreme points of the simplex $\simplex,$ and the second part shows that each such inequality is a core facet.

Suppose, for some $t \in \{0,\ldots, n-2\},$ that $\beta_0+\sum_{j=1}^n\beta_jx_j +\beta^{\prime}y \geq 0$ is a core facet that is satisfied exactly at $(\x,y)=(\x^k,\mk)$ for $k \in \{t,t+1,t+2\}.$ Then $d=2$ in the proof of \mythref{foundation15} since $d \ge 3,$ $d$ odd, was shown not possible, and $d \geq 4,$ $d$ even, was shown to yield a core facet of the form \eqref{trivial}. Clearly, the facet is not of the form \eqref{trivial} since $D(t,t+1) = 0$ gives $\beta_{t+1}=\beta^{\prime}m(\x^t) \neq 0.$ As a result, the first part of the proof of \mythref{foundation15} gives us that the facet must be of the form \eqref{trivia0}. The two equations in two unknowns $\beta_0$ and $\bar{\beta}$ that are obtained by evaluating $(\x,y)=(\x^k,\mk)$ for $k \in \{t,t+1\}$ within \eqref{trivia0} give $\beta_0=(n-1)$ and $\bar{\beta}=1$ as in \eqref{dominate2}, since $-m(\x^t)=m(\x^{t+1})=(-1)^{n-(t+1)}.$

The proof of \mythref{foundation15} showed inequalities \eqref{dominate2} to be valid for $\conv{G}$ and, for each $t \in \{0,\ldots, n-2\},$ the associated inequality to be satisfied exactly at $(\x,y)=(\x^k,\mk)$ for $k \in \{t, t+1, t+2\}.$ Thus, given such a $t,$ it is sufficient to identify $(n+1)$ affinely independent points in $\conv{G}$ that satisfy the inequality exactly. Consider the $n$ extreme points of $\conv{G},$ denoted by $(\x[p]^1,m(\x[p]^1)), \ldots, (\x[p]^n,m(\x[p]^n)),$ so that $\x[p]^i$ differs from $\x^{t+1}$ in only position $i.$ Since $\beta_1= \ldots = \beta_{t+1}$ and $(\x,y)=(\x^t,m(\x^t))$ satisfies the inequality exactly, \mythref{foundation45} gives us that $(\x,y)=(\x[p]^i,m(\x[p]^i))$ for $i \in \{1,\ldots, t+1\}$ satisfies the inequality exactly. Similarly, since $\beta_{t+2}= \ldots = \beta_n$ and $(\x,y)=(\x^{t+1},m(\x^{t+1}))$ satisfies the inequality exactly, \mythref{foundation45} gives us that $(\x,y)=(\x[p]^i,m(\x[p]^i))$ for $i \in \{t+2,\ldots, n\}$ satisfies the inequality exactly. Subtract $\x^{t+1}$ from every such point to reduce $\x[p]^i$ to $-2\boldsymbol{e}^i$ for $i \leq t+1$ and to $2\boldsymbol{e}^i$ for $i \geq t+2,$ where $\boldsymbol{e}^i$ is the unit vector in $\real^n$ having a 1 in position $i$ and $0$ elsewhere. Hence, $\x^{t+1},$ together with $\x[p]^1, \ldots, \x[p]^n,$ is an affinely independent set of points, so that $(\x^{t+1},m(\x^{t+1})),$ together with $(\x[p]^1,m(\x[p]^1)), \ldots, (\x[p]^n,m(\x[p]^n)),$  is an affinely independent set of points.
\end{proof}

\begin{proof}[\textbf{Proof of \mythref{resulting3}}]
Follows from \mythref{foundation900,foundation15,resulting}.
\end{proof}

\section{Summary and Open Questions}

This paper derives polyhedral results for the convex hull of symmetric multilinear polynomials (SMPs) taken over a box domain. \akshay{Exponential-sized extended formulations of general multilinear polynomials are available via the reformulation-linearization-technique (RLT), but symmetry and disjunctive programming enable a quadratic-sized extended formulation.} The goal of this paper is to obtain the convex hulls in the original variable spaces. Instead of adopting the tedious method of projecting the extended formulations, our approach is more elegant in the sense that we directly exploit the problem structure to define special core facets by which all facets can be characterized, and to then devise necessary and/or sufficient conditions on the coefficients of these facets. Whereas much of the theory is applicable to general SMPs over box constraints, we focus attention on two special problem classes: general supermodular (submodular) functions, and monomials having the variable lower bound equal to the negative of the upper bound. For each class, we use the necessary conditions to motivate families of core facets, and then prove that no other such facets can exist. \akshay{Our derivations of these convex hulls provides alternate proofs to those in literature.} For both classes, we use RLT results to characterize for each facet the set of all points at which the inequality is satisfied exactly.

A direction of future research is the identification of convex hull forms for more general families of SMPs than the two types within this paper, and likewise for the identification of all points within the resulting graphs that satisfy each facet exactly. \akshay{One open question in this regard is to generalize \mythref{resulting3} to monomials taken over $[-1,r]^{n}$ for arbitrary $r > 0$. Even more broadly, we do not know an explicit minimal description for convex hulls of general SMPs in the original variable space. Projecting the extended formulation of \mythref{extform} is an option, but this is likely to result in a combinatorial explosion as is usually the case when projecting from a higher-dimensional space. Another question that is open is whether similar to the two families of SMPs analysed in this paper, every SMP has a linear (or even polynomial) number of core facets. A positive answer to this question would imply a straightforward separation algorithm for the convex hull due to \mythref{sepcompl} without invoking the ellipsoid method and the optimization algorithm in \mythref{optcompl}.}

\paragraph{Acknowledgement.}
This research was initiated when the authors were in the School of Mathematical and Statistical Sciences at Clemson University, USA, during which the first two authors (YX and WA) were supported by ONR grant N00014-16-1-2168 and the third author (AG) was supported by ONR grant N00014-16-1-2725.

{
\newrefcontext[sorting=nyt]
\printbibliography[heading=bibliography]
}

\newpage
\renewcommand{\theequation}{\thesection.\arabic{equation}}
\setcounter{equation}{0}

\begin{appendices}

\section{Properties of $\conv{G}$ from RLT}	\label{sec:propG}

The key observation is that the equivalence of \eqref{miss1} between the polyhedral sets $P$ of \eqref{RLTstep2} and $\conv{T}$ of \eqref{miss1} continues to hold true with the inclusion of the restriction $y= m(\x)$  within these two sets. Consider a generalization of the set $G$ given by 
\begin{equation}
G^{\prime} \equiv\left\{(\x,y)\in \real^n \times \real: \x \in X^{\prime}, \; y= m(\x)\right\}, \label{marker}
\end{equation}
that is obtained by replacing $\x \in X$ of \eqref{Xdef} with $\x \in X^{\prime}$ of \eqref{Xprimedef}. Let
\begin{equation}
P_y\equiv\left\{(\x,\x[w],y) \in \real^n\times\real^{2^n-(n+1)}\times \real:(\x,\x[w]) \in P, \; y=\left\{m(\x)\right\}_L\right\}  \label{handy}
\end{equation}
be the set $P$ of \eqref{RLTstep2} that is modified to include the additional variable $y$ and additional restriction $y=\left\{m(\x)\right\}_L,$ and let 
\begin{equation}
T_y\equiv\left\{(\x,\x[w],y) \in \real^n\times\real^{2^n-(n+1)}\times \real:(\x,\x[w]) \in T, \; y=m(\x)\right\}  \nonumber
\end{equation}
be the set $T$ of \eqref{miss1} that is modified to include the additional variable $y$ and additional restriction $y=m(\x).$ Then we have that
\begin{equation}
P_y=conv\left(T_y \right). \label{pete}
\end{equation}
Equality \eqref{pete} holds true from \eqref{miss1} because the extreme points of $P$ and $P_y$ correspond in a one-to-one fashion in such a manner that $(\x,\x[w])$ is an extreme point of $P$ if and only if $(\x,\x[w],y)$ is an extreme point of $P_y$ having $y=\left\{m(\x)\right\}_L=m(\x).$ Therefore, every extreme point of the polytope $P_y$ is in $T_y$ and $T_y \subseteq P_y$ by construction. 

The three propositions below are consequences of \eqref{pete}, where $\mbox{Proj}_{(\x,y)}(\bullet)$ denotes the projection of the set $\bullet$ onto the space of the variables $(\x,y).$ 

\begin{proposition}	\thlabel{equalities}
$\conv{G^{\prime}} = \conv{\left(\mbox{Proj}_{(\x,y)}(T_y)\right)} = \mbox{Proj}_{(\x,y)}\left(\conv{T_y}\right) = \mbox{Proj}_{(\x,y)}\left(P_y\right)$. 
\end{proposition}
\begin{proof}
The first equality follows because $G^{\prime}=\mbox{Proj}_{(\x,y)}(T_y),$ the second equality follows from interchanging the projection and convex hull operators, and the third equality follows from \eqref{pete}.
\end{proof} 

\begin{proposition}	\thlabel{equalities21}
$conv\left(G^{\prime}\right)$ is a polytope with $2^n$ extreme points that have one-to-one correspondence with the extreme points of $X^{\prime}$ in such a manner that $y=m(\x)$ at each such point $\x.$ 
\end{proposition}
\begin{proof}
$conv\left(G^{\prime}\right)$ is a polytope with no more than $2^n$ extreme points since it is the projection of the polytope $P_y$ having $2^n$ extreme points onto the $(\x,y)$ space, as stated in \mythref{equalities}. However, each of the $2^n$ points of $conv\left(G^{\prime}\right)$ satisfying the property that every entry of $\x$ has either $x_j=L_j$ or $x_j=U_j$ for all $j \in N,$ and that $y=m(\x),$ is trivially an extreme point of $conv\left(G^{\prime}\right).$
\end{proof} 

\mythref{equalities21} also follows from results in \citet{rikun1997convex}. As opposed to using the projection from a higher-dimensional RLT space as in the above proof, \citeauthor{rikun1997convex} shows that every $(\x,y) \in G^{\prime}$ which has some $x_j \in (L_j,U_j)$ can be expressed as a strict convex combination of two distinct points in $G^{\prime}.$

The following result addresses the validity of a linear inequality for $conv\left(G^{\prime}\right)$ in terms of the restrictions of the set $P_y.$

\begin{corollary}	\thlabel{equalities24}
A linear inequality $\beta_0+\sum_{j=1}^n\beta_jx_j+\beta^{\prime}y \geq 0$ is valid for $conv\left(G^{\prime}\right)$ if and only if it can be uniquely expressed as a linear combination of the restrictions of $P_y$ using nonnegative multipliers $\x[\pi] \in \real^{2^n}$ and the scalar $\beta^{\prime},$ so that
\begin{equation}
\beta_0+\sum_{j=1}^n\beta_jx_j+\beta^{\prime}y=\sum_{K \subseteq N}\pi_{K}F_{K}(x)+\beta^{\prime}\left(y-\left\{m(\x)\right\}_L\right). \label{ruby}
\end{equation}
\end{corollary}
\begin{proof}
The existence of nonnegative multipliers $\x[\pi]$ satisfying \eqref{ruby} follows from the result of \mythref{equalities} that $\conv{G^{\prime}}\mbox{Proj}_{(\x,y)}\left(P_y\right),$ as then a linear inequality will be valid for $\conv{G^{\prime}}$ if and only if it is valid for $P_y.$ The uniqueness follows from the invertibility of the matrix $\scriptsize\left[\begin{array}{cc} U_1 &  -1 \\ -L_1 & 1 \end{array}\right] \normalsize \otimes \ldots \otimes \scriptsize \left[\begin{array}{cc} U_n &  -1 \\ -L_n & 1 \end{array}\right]$ of \eqref{ca4}.
\end{proof}

\section{Illustrative Examples}

\subsection{Use of Kronecker Products}

The example below illustrates the use of Kronecker products to establish the two stated properties of $P,$ to demonstrate \eqref{miss1}, and to provide insights into the two observations.
\begin{example}	\label{ex1}
Consider $X^{\prime}$ of \eqref{Xprimedef} with $N=\{1,2,3\}$ in the $n=3$ variables $x_1, x_2,$ and $x_3.$ The $2^3=8$ inequalities of \eqref{RLTstep} are expressed in \eqref{prodkron2} as 
\begin{eqnarray} \scriptsize
\left[\begin{array}{rr|rr|rr|rr}   
U_1U_2U_3 & -U_1U_2 & -U_1U_3 & U_1 & -U_2U_3 & U_2 & U_3 & -1 \\  
-U_1U_2L_3 & U_1U_2 & U_1L_3 & -U_1 & U_2L_3 & -U_2 & -L_3 & 1 \\ \cline {1-8}
-U_1L_2U_3 & U_1L_2 & U_1U_3 & -U_1 & L_2U_3 & -L_2 & -U_3 & 1 \\  
U_1L_2L_3 & -U_1L_2 & -U_1L_3 & U_1 & -L_2L_3 & L_2 & L_3 & -1 \\ \cline {1-8}
-L_1U_2U_3 & L_1U_2 & L_1U_3 & -L_1 & U_2U_3 & -U_2 & -U_3 & 1 \\  
L_1U_2L_3 & -L_1U_2 & -L_1L_3 & L_1 & -U_2L_3 & U_2 & L_3 & -1 \\ \cline {1-8}
L_1L_2U_3 & -L_1L_2 & -L_1U_3 & L_1 & -L_2U_3 & L_2 & U_3 & -1 \\  
-L_1L_2L_3 & L_1L_2 & L_1L_3 & -L_1 & L_2L_3 & -L_2 & -L_3 & 1 \\
\end{array}\right]\left(\begin{array}{cccccccc} 1 \\ x_3 \\ \cline {1-1} x_2 \\ x_2x_3 \\ \cline {1-1} x_1 \\ x_1x_3 \\ \cline {1-1} x_1x_2 \\ x_1x_2x_3\\ \end{array} \right) \geq \left(\begin{array}{cccccccc} 0 \\ 0 \\  \cline {1-1} 0 \\ 0 \\  \cline {1-1} 0 \\ 0 \\  \cline {1-1} 0 \\ 0\\ \end{array} \right). \label{larger}
\end{eqnarray} 
The equations of \eqref{ca6} in nonnegative variables $\boldsymbol{\lambda}$ take the form
\begin{eqnarray} \tiny
\left(\begin{array}{cccccccc} 1 \\ x_3 \\ \cline {1-1} x_2 \\ w_{23} \\ \cline {1-1} x_1 \\ w_{13} \\ \cline {1-1} w_{12} \\ w_{123}\\ \end{array} \right) =  \frac{1}{d_1d_2d_3} \left[\begin{array}{rr|rr|rr|rr}   
1 & 1 & 1 & 1 & 1 & 1 & 1 & 1 \\  
L_3 & U_3 & L_3 & U_3 & L_3 & U_3 & L_3 & U_3 \\ \cline {1-8}
L_2 & L_2 & U_2 & U_2 & L_2 & L_2 & U_2 & U_2 \\  
L_2L_3 & L_2U_3 & U_2L_3 & U_2U_3 & L_2L_3 & L_2U_3 & U_2L_3 & U_2U_3 \\ \cline {1-8}
L_1 & L_1 & L_1 & L_1 & U_1 & U_1 & U_1 & U_1 \\  
L_1L_3 & L_1U_3 & L_1L_3 & L_1U_3 & U_1L_3 & U_1U_3 & U_1L_3 & U_1U_3 \\ \cline {1-8}
L_1L_2 & L_1L_2 & L_1U_2 & L_1U_2 & U_1L_2 & U_1L_2 & U_1U_2 & U_1U_2 \\  
L_1L_2L_3 & L_1L_2U_3 & L_1U_2L_3 & L_1U_2U_3 & U_1L_2L_3 & U_1L_2U_3 & U_1U_2L_3 & U_1U_2U_3 \\
\end{array}\right]\left(\begin{array}{cccccccc} \lambda_1 \\ \lambda_2 \\ \cline {1-1} \lambda_3 \\ \lambda_4 \\ \cline {1-1} \lambda_5 \\ \lambda_6 \\ \cline {1-1} \lambda_7 \\ \lambda_8\\ \end{array} \right), \label{sume}
\end{eqnarray}
where the RLT linearization step sets $w_{12}=x_1x_2,$ $w_{13}=x_1x_3,$ $w_{23}=x_2x_3,$ and $w_{123}=x_1x_2x_3.$ 

There are eight extreme points to this system, with extreme point $j$ having $\lambda_j=d_1d_2d_3$ and $\lambda_i=0$ for $i \neq j.$ Then the eight extreme points to $P$ of \eqref{RLTstep2} and \eqref{ca4} are given by the eight columns of the above $8 \times 8$ matrix, less the first row. Consequently, we have that $P=conv(T)$ with 
\begin{equation}
T=\tiny\left\{ \left(\begin{array}{c} x_1 \\  x_2 \\ x_3 \\  w_{12} \\ w_{13} \\ w_{23} \\ w_{123}\\ \end{array} \right):  L_j \leq x_j \leq U_j \; \forall \; j=1,2,3, \; \left(\begin{array}{c} x_1x_2 \\ x_1x_3 \\ x_2x_3 \\ x_1x_2x_3 \\ \end{array} \right)=\left(\begin{array}{c} w_{12} \\ w_{13} \\ w_{23} \\ w_{123} \\ \end{array} \right)\right\}, \nonumber
\end{equation}
as in \eqref{miss1}.

Now consider any multilinear polynomial $\sum_{J \subseteq N}\alpha_J\prod_{j \in J}x_j$ as found in \mythref{obs1,obs2}, and the corresponding system in variables $\boldsymbol{\pi}\in \real^{2^n}.$ Using obvious notation, denote \eqref{larger} by $A\boldsymbol{v} \geq \boldsymbol{0}$ so that the (scaled) $8 \times 8$ matrix of \eqref{sume} is $A^{-1}.$ Observe that the eight inequalities of \eqref{larger} correspond, in order, to the functions $F(K)$ of \eqref{Fun} having $K= \emptyset, \{3\}, \{2\}, \{2,3\}, \{1\}, \{1,3\}, \{1,2\}, \{1,2,3\}.$ (This order coincides with the variable indices of the products within the vector $\boldsymbol{v}.$) Accordingly define $\boldsymbol{\alpha}^T=(\alpha_{\emptyset},\alpha_3,\alpha_2,\alpha_{23},\alpha_1,\alpha_{13},\alpha_{12},\alpha_{123}),$ and $\boldsymbol{\pi} \in \real^{8}$ by $\boldsymbol{\pi}^T=(\pi_{\emptyset},\pi_3,\pi_2,\pi_{23},\pi_1,\pi_{13},\pi_{12},\pi_{123})$ so that $\boldsymbol{\pi}^T=\boldsymbol{\alpha}^TA^{-1}.$ Then we have
\begin{equation}
\sum_{J \subseteq N} \alpha_J\prod_{j \in J}x_j=\boldsymbol{\alpha}^T\boldsymbol{v}=(\boldsymbol{\alpha}^TA^{-1})(A\boldsymbol{v})=\boldsymbol{\pi}^T(A\boldsymbol{v}) =\sum_{K \subseteq N}\pi_KF(K), \label{includee}
\end{equation}
where the four equalities follow, from left to right, from the definitions of $\boldsymbol{\alpha}$ and $\boldsymbol{v},$ the multiplicative inverse of $A,$ the definition of $\boldsymbol{\pi},$ and the stated equivalence between the vector $A\boldsymbol{v}$ and the functions $F(K)$ of \eqref{Fun}. The computing of the multipliers $\pi_K$ by \mythref{obs2} follows from \eqref{includee} since each entry of the vector $\boldsymbol{\alpha}^TA^{-1}$ corresponds to a distinct $K \subseteq N,$ and realizes value $\frac{1}{D_3}(\sum_{J \subseteq N} \alpha_J\prod_{j \in J}\hat{x}_j),$ where $\hat{x}_j=U_j$ for all $j \in K$ and $\hat{x}_j=L_j$ for all $j \notin K.$ For each such $K,$ this value is the multiplier $\pi_K$ on the associated function $F(K).$ \mythref{obs2} gives us, provided it is nonnegative over the extreme points of $X^{\prime},$ that the polynomial $\sum_{J \subseteq N}\alpha_J\prod_{j \in J}x_j$ vanishes at a point $\hat{\x} \in X^{\prime}$ if and only if $\pi_KF(K)=0$ for all $K \subseteq N,$ where each $F(K)$ is evaluated at $\hat{\x}.$ \hfill$\diamond$
\end{example}

\subsection{Convex Hull via Projection}

The example below illustrates the computation of the set $conv\left(G^{\prime}\right)$ via the projection from the extended variable space $(\x,\boldsymbol{w},y)$ of $P_y$ onto the space of the variables $(\x,y),$ as stated in \mythref{equalities}. For simplicity, the sets $X^{\prime}$ and $G^{\prime}$ of \eqref{Xprimedef} and \eqref{marker} are reduced to the sets $X$ and $G$, respectively, by setting $L_j=\ell$ and $U_j =u$ for all $j \in N.$ The example also demonstrates the result of \mythref{obs2} that allows for the identification of those points within $X$ at which each facet is satisfied exactly. Example~\ref{ex2} builds upon Example~\ref{ex1}, and will be later referenced.

\begin{example}	\label{ex2}
Consider the set $G$ with $N=\{1,2,3\}$ in the $n=3$ variables $x_1, x_2,x_3,$ and the variable $y,$ with $y=m(\x)=x_1x_2x_3,$. The set $P_y$ of \eqref{handy}, whose projection onto the $(\x,y)$ variable space gives $\conv{G}$ as stated in \mythref{equalities}, is expressed in matrix form below. The matrix partitioning is used to emphasize that the first eight restrictions are inequalities and the last restriction is equality.
\begin{align}
&\scriptsize\left[\begin{array}{rr|rr|rr|rr|r}   
u^3 & -u^2 & -u^2 & u & -u^2 & u & u & -1 & 0 \\  
-\ell u^2 & u^2 & \ell u & -u & \ell u & -u & -\ell & 1 & 0 \\ \cline {1-9}
-\ell u^2 & \ell u & u^2 & -u & \ell u & -\ell & -u & 1 & 0 \\  
\ell^2 u & -\ell u & -\ell u & u & -\ell^2 & \ell & \ell & -1 & 0 \\ \cline {1-9}
-\ell u^2 & \ell u & \ell u & -\ell & u^2 & -u & -u & 1 & 0 \\  
\ell^2 u & -\ell u & -\ell^2 & \ell & -\ell u & u & \ell & -1 & 0 \\ \cline {1-9}
\ell^2 u & -\ell^2 & -\ell u & \ell & -\ell u & \ell & u & -1 & 0 \\  
-\ell^3 & \ell^2 & \ell^2 & -\ell & \ell^2 & -\ell & -\ell & 1 & 0 \\ 
\end{array}\right]\left(\begin{array}{cccccccc} 1 \\ x_3 \\ \cline {1-1} x_2 \\ w_{23} \\ \cline {1-1} x_1 \\ w_{13} \\ \cline {1-1} w_{12} \\ w_{123} \\  \end{array} \right) \geq \left(\begin{array}{cccccccc} 0 \\ 0 \\  \cline {1-1} 0 \\ 0 \\  \cline {1-1} 0 \\ 0 \\  \cline {1-1} 0 \\ 0 \\  \end{array} \right)\nonumber  \\
&\scriptsize\phantom{|}\left[\begin{array}{rr|rr|rr|rr|r}   
\phantom{\ell u^2|}0 & \phantom{-\ell |}0 & \phantom{\ell u}0 & \phantom{-|}0\phantom{l} & \phantom{\ell u}0 & \phantom{-\ell}0 & \phantom{a|}0 & -1 &  1\\ 
\end{array}\right]\phantom{|}\left(\begin{array}{cccccccc}  \phantom{aa}y\phantom{aa} \\ \end{array} \right) \phantom{|}=\phantom{|} \left(\begin{array}{cccccccc} \phantom{|}0\phantom{|} \\  \end{array} \right) \label{largeee}
\end{align}

\noindent The eight inequalities are the linearized form of \eqref{larger} when $L_1,$ $L_2,$ and $L_3$ are set to $\ell,$ and when $U_1,$ $U_2,$ and $U_3$ are set to $u$, while the equation is $-\{m(\x)\}_L+y=0.$ Since we desire to project \eqref{largeee} onto the space of the variables $(x_1,x_2,x_3,y),$ the projection cone takes the form
\begin{eqnarray}  
\scriptsize
\left[\begin{array}{rrrrrrrrr}   
u & -u & -u & u & -\ell & \ell & \ell & -\ell & 0\\  
u & -u & -\ell & \ell & -u & u & \ell & -\ell & 0\\
u & -\ell & -u & \ell & -u & \ell & u & -\ell & 0\\
-1 & 1 & 1 & -1 & 1 & -1 & -1 & 1 & -1\\
\end{array}\right]\left(\begin{array}{cccccccc} \pi_{\emptyset} \\ \pi_3 \\ \pi_2 \\ \pi_{23} \\ \pi_1 \\ \pi_{13} \\ \pi_{12} \\ \pi_{123} \\ \beta^{\prime} \\ \end{array} \right) = \left(\begin{array}{ccc} 0 \\ 0 \\  0\\ 0\\ \end{array} \right), \label{largend}
\end{eqnarray}
where $\boldsymbol{\pi}$ are the nonnegative multipliers on the eight inequality restrictions of \eqref{largeee}, and where $\beta^{\prime}$ is the multiplier on the equation. When $\ell > 0,$ there are fifteen extreme directions to this cone, with six ``trivial directions" resulting in the six facets $x_j \geq \ell$ and $-x_j \geq -u$ for $j \in \{1,2,3\}.$ Specifically, setting $\beta^{\prime}=0$ for each $j,$ the ``trivial direction" having $\pi_K=1$ if $ j \in K$ and $\pi_K=0$ otherwise gives $x_j- \ell \geq 0,$ and the ``trivial direction" having $\pi_K=1$ if $j \in (N-K)$ and $\pi_K=0$ otherwise gives $u-x_j \geq 0,$ all inequalities scaled by $\frac{1}{(u-\ell)^2}.$ (Clearly, for each $j \in \{1,2,3\},$ the inequality $x_j - \ell \geq 0$ is satisfied exactly at all points $(x_1,x_2,x_3,y) \in G$ having $x_j = \ell,$ and the inequality $u-x_j \geq 0$ is satisfied exactly at all points $(x_1,x_2,x_3,y) \in G$ having $x_j = u.$) Each of the remaining nine directions is depicted as a column of the below matrix.
\begin{eqnarray} 
\scriptsize
\left[\begin{array}{ccccccccc}   
0 & 0 & 0 & 0 & 0 & 0 & 0 & \ell & \ell+2u\\  
\ell+u & \ell & \ell+u & \ell & 0 & 0 & 0 & 0 & u\\
\ell & \ell+u & 0 & 0 & \ell+u & \ell & 0 & 0 & u\\
\ell+u & \ell+u & u & 0 & u & 0 & \ell & 0 & 0\\
0 & 0 & \ell & \ell+u & \ell & \ell+u & 0 & 0 & u\\
u & 0 & \ell+u & \ell+u & 0 & u & \ell & 0 & 0\\
0 & u & 0 & u & \ell+u & \ell+u & \ell & 0 & 0\\
0 & 0 & 0 & 0 & 0 & 0 & 2\ell+u & u & 0\\
\ell-u & \ell-u & \ell-u & \ell-u & \ell-u & \ell-u & -\ell+u & -\ell+u & -\ell+u\\
\end{array}\right] \nonumber
\end{eqnarray}
The set $\conv{G}$ is then defined in terms of the $x_j - \ell \geq 0$ and $u-x_j \geq 0$ restrictions for $j \in \{1,2,3\},$ together with the nine facets that are listed in the first column of the below table, upon dividing each inequality by $(u-\ell).$

\begin{table}
\begin{center}
\caption{Facet and Points in $G$ where Satisfied Exactly when $\ell >0$}
\label{tab1}
\begin{tabular}{|c | c c c |}
\hline
\vspace{-.1 in}
& & & \\
\vspace{-.1 in}
Facet &  \multicolumn{3}{c} {Points in $G$ where Satisfied Exactly}  \\ 
& & & \\
\hline
$(\ell^2)x_1+(\ell u)x_2+(u^2)x_3-y \geq \ell u(\ell+u)$     &$(x_1,\ell,\ell)$   &$(u,x_2,\ell)$    &$(u,u,x_3)$        \\ 
$(\ell^2)x_1+(u^2)x_2+(\ell u)x_3-y \geq \ell u(\ell+u)$     &$(x_1,\ell,\ell)$   &$(u,x_2,u)$    &$(u,\ell,x_3)$       \\
$(\ell u)x_1+(\ell^2)x_2+(u^2)x_3-y \geq \ell u(\ell+u)$     &$(x_1,u,\ell)$   &$(\ell,x_2,\ell)$    & $(u,u,x_3)$      \\
$(u^2)x_1+(\ell^2)x_2+(\ell u)x_3-y \geq \ell u(\ell+u)$     &$(x_1,u,u)$   &$(\ell,x_2,\ell)$    & $(\ell,u,x_3)$     \\
$(\ell u)x_1+(u^2)x_2+(\ell^2)x_3-y \geq \ell u(\ell+u)$     &$(x_1,\ell,u)$   &$(u,x_2,u)$    & $(\ell,\ell,x_3)$       \\
$(u^2)x_1+(\ell u)x_2+(\ell^2)x_3-y \geq \ell u(\ell+u)$     &$(x_1,u,u)$   &$(\ell,x_2,u)$    & $(\ell,\ell,x_3)$      \\
$-(\ell^2)x_1-(\ell^2)x_2-(\ell^2)x_3+y \geq -2\ell^3$     &$(x_1,\ell,\ell)$   &$(\ell,x_2,\ell)$    & $(\ell,\ell,x_3)$      \\
$-(\ell u)x_1-(\ell u)x_2-(\ell u)x_3+y \geq -\ell u(\ell+u)$     & \multicolumn{3}{c|} {Any $x_i=\ell,$ any $x_j=u, \; i \neq j.$}  \\
$-(u^2)x_1-(u^2)x_2-(u^2)x_3+y \geq -2u^3$   &$(x_1,u,u)$   &$(u,x_2,u)$    & $(u,u,x_3)$    \\
\hline
\end{tabular}
\end{center}
\end{table}

By letting $y=m(\x)$ within each facet of Table~\ref{tab1}, \mythref{obs2} allows us to identify, in terms of the positive multipliers $\boldsymbol{\pi}$ of \eqref{largend}, the points in $G$ where each such inequality is satisfied exactly. Definition \eqref{Fun} gives us, for this example,  that a function $F(K)$ vanishes at a point $\hat{\x} \in X$ if and only if either $\hat{x}_j=\ell$ for some $j \in K$ or $\hat{x}_j=u$ for some $j \notin K.$ The second column of Table~\ref{tab1}, which then logically follows, lists the set of all points $(x_1,x_2,x_3,y) \in G$ where each facet is satisfied exactly, with the value $y=x_1x_2x_3$ suppressed for simplicity. The notation $``x_j"$ found within this table indicates that the variable $x_j$ can have $\ell \leq x_j \leq u.$ Observe that, given any realization of $(x_1,x_2,x_3)$ wherein two of the variables have values at either their lower or upper bounds, the first six facets enforce $y \leq x_1x_2x_3$ while the last three facets enforce $y \geq x_1x_2x_3,$ ensuring that $y=x_1x_2x_3.$

For the case when $\ell=0,$ each of the first eight facets of Table~\ref{tab1} is twice listed by repetition, and the inequalities $x_j \geq 0$ for $j \in \{1,2,3\}$ are not facets. A statement of the five facets resulting from Table~\ref{tab1} and those points $(x_1,x_2,x_3,y) \in G,$ with $y=x_1x_2x_3$ suppressed for simplicity, where each facet is satisfied exactly is given in Table~\ref{tab2}. Here, the notation $``x_j"$ indicates that the variable $x_j$ can have $0 \leq x_j \leq u.$ \hfill$\diamond$

\begin{table}
\begin{center}
\caption{Facet and Points in $G$ where Satisfied Exactly when $\ell =0$}
\label{tab2}
\begin{tabular}{|c | c c c |}
\hline
\vspace{-.1 in}
& & & \\
\vspace{-.1 in}
Facet &  \multicolumn{3}{c} {Points in $G$ where Satisfied Exactly}  \\ 
& & & \\
\hline
$(u^2)x_3-y \geq 0$     &$(x_1,x_2,0)$   & $(u,u,x_3)$   &     \\ 
$(u^2)x_2-y \geq 0$     &$(x_1,0,x_3)$   &  $(u,x_2,u)$ &      \\
$(u^2)x_1-y \geq 0$     &$(0,x_2,x_3)$   &   $(x_1,u,u)$ &     \\
$y \geq 0$     & \multicolumn{3}{c|} {Any $x_i=0$}   \\
$-(u^2)x_1-(u^2)x_2-(u^2)x_3+y \geq -2u^3$   &$(x_1,u,u)$   &$(u,x_2,u)$    & $(u,u,x_3)$    \\
\hline
\end{tabular}
\end{center}
\end{table}

\end{example}

\subsection{Convex Hull for Supermodular Function}

\begin{example} \label{ex3}
Consider $G$ in the $n=3$ variables $x_1, x_2,x_3,$ and the variable $y,$ with $m(\x)=x_1x_2x_3$ as in Example~\ref{ex2}, and with $\ell \geq 0$. Inequalities \eqref{rest0} hold true because $(m(\x^k)-m(\x^{k-1}))$ takes values $\ell^2(u-\ell),$ $(\ell u)(u-\ell),$ and $u^2(u-\ell)$ when $k=1,$ $2,$ and $3,$ respectively. Then \mythref{result1} is applicable, and core facet \eqref{res1} is
\begin{equation}
-\ell u(\ell +u)+(\ell^2)x_1+(\ell u)x_2+(u^2)x_3-y \geq 0, \label{testing1}
\end{equation}
and core facets \eqref{res2} are
\begin{eqnarray}
-\ell^2u-\ell^2\left(x_1+x_2+x_3-[u+2\ell]\right)+y &\geq& 0,\label{testing2}\\
-\ell u^2-\ell u\left(x_1+x_2+x_3-[2u+\ell]\right)+y &\geq& 0,\label{testing3}\\
-u^3-u^2\left(x_1+x_2+x_3-[3u]\right)+y &\geq& 0. \label{testing4}
\end{eqnarray}
For $\ell>0,$ $m(\x)$ is strictly supermodular and facet \eqref{testing1}, together with the additional five facets obtained by permuting the coefficients $\x[\beta],$ are the first six facets of Table~\ref{tab1}. Facets \eqref{testing2}--\eqref{testing4} are the last three facets of Table~\ref{tab1}. In addition, the six inequalities $\ell \leq x_j \leq u$ for $j\in \{1,2,3\}$ comprise the remaining facets of $\conv{G}.$ For $\ell=0,$ $m(\x)$ is supermodular (but not strictly supermodular because $(m(\x^k)-m(\x^{k-1}))$ takes values $0,$ $0,$ and $u^3$ when $k=1,$ $2,$ and $3,$ respectively), and three of the six facets obtained by permuting $\x[\beta]$ in \eqref{testing1} are repetitive. The three resulting facets are the first three inequalities of Table~\ref{tab2}. Also, and consistent with Remark 1, facets \eqref{testing2} and \eqref{testing3} are the same. Facets \eqref{testing2} and \eqref{testing4} are the last two inequalities of Table~\ref{tab2}. In addition for $\ell=0,$  $x_j \geq 0$ for $j \in \{1,2,3\}$ are not facets.

Suppose that $G$ in the $n=3$ variables $x_1, x_2,x_3,$ and the variable $y,$ is changed to have $y=m(\x)=(2u-\ell)(x_1x_2+x_1x_3+x_2x_3)-x_1x_2x_3,$ with arbitrary $\ell$ and $u.$ Inequalities \eqref{rest0} hold true because the difference between the two sides takes values $2(u-\ell)^3$ and $(u-\ell)^3$ when $k=1$ and $k=2,$ respectively, so that \mythref{result1} is applicable with $m(\x)$ strictly supermodular. Core facet \eqref{res1} is
\begin{equation}
(\ell u)(4\ell -5u)+(4\ell u-3\ell^2)x_1+\left(2u^2-\ell^2\right)x_2+\left(3u^2-2\ell u\right)x_3-y \geq 0, \nonumber
\end{equation}
and core facets \eqref{res2} are
\begin{eqnarray*}
-(4\ell u^2-\ell^3 -\ell^2 u)-(4\ell u - 3\ell^2)\left(x_1+x_2+x_3-[u+2\ell]\right)+y &\geq& 0,\\
-(2\ell u^2-2\ell^2 u+2u^3)-(2u^2-\ell^2)\left(x_1+x_2+x_3-[2u+\ell]\right)+y &\geq& 0,\\
-(5u^3-3\ell u^2)-(3u^2-2\ell u)\left(x_1+x_2+x_3-[3u]\right)+y &\geq& 0.\\
\end{eqnarray*}
There are exactly nine facets \eqref{sum24} with $\beta^{\prime}= \pm 1$ for $\conv{G},$ including the four listed above and the five additional obtained by permuting the coefficients $\x[\beta]$ in the first inequality. The six inequalities $\ell \leq x_j \leq u$ for $j\in \{1,2,3\}$ comprise the remaining facets of $\conv{G},$ as $\ell <u.$ \hfill$\diamond$
\end{example}

\subsection{Convex Hull for Monomial}

\begin{example} \label{ex4}
Consider $G$ in the $n=3$ variables $x_1,x_2,x_3,$ and the variable $y,$ with $m(\x)=5x_1x_2x_3,$ and with $-\ell=u=2$. Since $n \geq 3,$ the convex hull representation is given by the (scaled) six inequalities $-2 \leq x_j \leq 2$ for $j \in \{1,2,3\}$ found in \eqref{con258}, together with the remaining (scaled) two inequalities of \eqref{con258}, $-40 \leq x_4 \leq 40,$ and the eight inequalities of \eqref{con269}, with the last ten such inequalities summarized in Table~\ref{tab3} below. Here, $y$ is used in lieu of $x_4$ for clarity. Each inequality $\beta_0 +\sum_{j=1}^n\beta_jx_j+\beta^{\prime}y \geq 0$ has been scaled to $\beta_0 +\sum_{j=1}^n\left(\frac{\beta_j}{u}\right)x_j+\left(\frac{\beta^{\prime}}{c_nu^n}\right)y \geq 0$ as discussed at the beginning of this section to handle the coefficient $c_n=5$ found in $m(\x)$ and the variable bounds $-\ell=u=2,$ and has then been rescaled to have integer coefficients. The first column of the table gives the facet and the second column, explained later in Example~\ref{ex6}, gives the set of all points $(x_1,x_2,x_3,y) \in G$ where each facet is satisfied exactly. Consistent with Tables 1 and 2, the notation $``x_j"$ indicates that the variable $x_j$ can have $-2 \leq x_j \leq 2,$ and the value $y=x_1x_2x_3$ is suppressed for simplicity.\hfill$\diamond$

\begin{table}
\begin{center}
\caption{Facet and Points in $G$ where Satisfied Exactly}
\label{tab3}
\begin{tabular}{|c | c c c|}
\hline
\vspace{-.1 in}
& & & \\
\vspace{-.1 in}
Facet &  \multicolumn{3}{c} {Points in $G$ where Satisfied Exactly}  \\ 
& & &\\
\hline
$40-y \geq 0$     & \multicolumn{3}{c|} {$(2,2,2)\;\;\;\;(2,-2,-2)\;\;\;\;(-2,2,-2)\;\;\;\;(-2,-2,2)$}     \\
$40+y \geq 0$     & \multicolumn{3}{c|} {$(-2,-2,-2)\;\;\;\;(-2,2,2)\;\;\;\;(2,-2,2)\;\;\;\;(2,2,-2)$}     \\
\hline
$80+20x_1+20x_2+20x_3-y\geq 0$     &$(x_1,-2,-2)$   &$(-2,x_2,-2)$ &$(-2,-2,x_3)$\\
\hline
$80-20x_1-20x_2-20x_3+y\geq 0$     &$(x_1,2,2)$   &$(2,x_2,2)$ &$(2,2,x_3)$\\
\hline
$80-20x_1+20x_2+20x_3+y\geq 0$     &$(x_1,-2,-2)$   &$(2,x_2,-2)$ &$(2,-2,x_3)$\\
$80+20x_1-20x_2+20x_3+y\geq 0$     &$(x_1,2,-2)$   &$(-2,x_2,-2)$ &$(-2,2,x_3)$\\
$80+20x_1+20x_2-20x_3+y\geq 0$     &$(x_1,-2,2)$   &$(-2,x_2,2)$ &$(-2,-2,x_3)$\\
\hline
$80-20x_1-20x_2+20x_3-y\geq 0$     &$(x_1,2,-2)$   &$(2,x_2,-2)$ &$(2,2,x_3)$\\
$80-20x_1+20x_2-20x_3-y\geq 0$     &$(x_1,-2,2)$   &$(2,x_2,2)$ &$(2,-2,x_3)$\\
$80+20x_1-20x_2-20x_3-y\geq 0$     &$(x_1,2,2)$   &$(-2,x_2,2)$ &$(-2,2,x_3)$\\
\hline
\end{tabular}
\end{center}
\end{table}

\end{example}

\subsection{Exactness for Supermodular Functions}

\begin{example} \label{ex5}
Consider $G$ in the $n=3$ variables $x_1,x_2,x_3,$ and the variable $y,$ with $m(\x)=x_1x_2x_3$ as introduced in Example~\ref{ex2} and revisited in Example 3 with $\ell \geq 0.$ As noted in Example~\ref{ex3}, \mythref{result1} is applicable, and \eqref{res1} takes the form \eqref{testing1}, while \eqref{res2} takes the forms \eqref{testing2}, \eqref{testing3}, and \eqref{testing4}. For $\ell >0,$ $m(\x)$ was noted to be strictly supermodular. Then \mythref{exactstrictsupermod} gives us inequality \eqref{testing1} is satisfied exactly at only those points $(\x,y) \in G$ where $\x$ is of the form $(x_1,\ell,\ell),$ $(u,x_2,\ell),$ or $(u,u,x_3),$ matching the first inequality of Table~\ref{tab1}. The next five  inequalities of Table~\ref{tab1}, which follow from permutations of the coefficients of \eqref{testing1}, have the set of points at which each is satisfied exactly obtained via the same permutations. \mythref{exactstrictsupermod} gives us that inequalities \eqref{testing2}, \eqref{testing3}, and \eqref{testing4} are satisfied exactly at only those points $(\x,y) \in G$ where $\x$ has, respectively, at least two entries of value $\ell,$ at least one entry of value $u$ and at least one entry of value $\ell,$ and at least two entries of value $u.$ These results of \mythref{exactstrictsupermod} match the last three rows of Table~\ref{tab1}. Also as noted in Example~\ref{ex3}, the six inequalities $\ell \leq x_j \leq u$ for $j\in \{1,2,3\}$ are all facets of $\conv{G}.$ For each of these last six facets, the set of points $(\x,y) \in G$ which satisfy it exactly is obvious.  For $\ell =0,$ $m(\x)$ was noted to be supermodular, but not strictly supermodular. In this case, \mythref{newsuper100} gives us inequality \eqref{testing1} is satisfied exactly at only those points $(\x,y) \in G$ having either $x_3=0$ or $x_1=x_2=u,$ matching the first inequality of Table~\ref{tab2}. The next two inequalities of Table~\ref{tab2}, which follow from permutations of the coefficients of \eqref{testing1}, have the set of points at which each is satisfied exactly obtained via the same permutations. Also as shown in Example~\ref{ex3}, inequalities \eqref{res2} are the last two inequalities of Table~\ref{tab2}. Then \mythref{newsuper100} gives us that the $y \geq 0$ inequality of \eqref{testing2} and the $-(u^2)x_1-(u^2)x_2-(u^2)x_3+y \geq -2u^3$ inequality of \eqref{testing4} are satisfied exactly at only those points $(\x,y) \in G$ where $\x$ has, respectively, at least one entry of value $0,$ and at least two entries of value $u,$ matching Table~\ref{tab2}.\hfill$\diamond$
\end{example}

\subsection{Exactness for Monomial}

\begin{example} \label{ex6}
Reconsider $G$ in the $n=3$ variables $x_1,x_2,x_3,$ with $m(\x)=5x_1x_2x_3,$ and with $-\ell=u=2$, as found in Example~\ref{ex4} and Table~\ref{tab3}. Recall that each facet within Table~\ref{tab3} has been suitably scaled to handle the coefficient $c_n=5$ found in $m(\x)$ and the variable bounds $-\ell=u=2.$ Therefore, consistent with the discussion at the beginning of this subsection, every point $(\tilde{\x},\tilde{y})$ identified in \mythref{finalet} as satisfying an inequality of the form \eqref{trivial}, \eqref{pair1}, \eqref{pair2}, or \eqref{dominate2} exactly for the case in which $c_n=1$ and $-\ell=u=1$ must be scaled to $(u\tilde{\x},c_nu^n\tilde{y})=(2\tilde{\x},40\tilde{y}).$ The second column of Table~\ref{tab3} gives these scaled points for each such inequality, with the value $y=5x_1x_2x_3$ suppressed for simplicity. Here, the two inequalities in the first block of constraints within Table~\ref{tab3} are scaled \eqref{trivial} for $\beta^{\prime}=-1$ and $\beta^{\prime}=1,$ respectively, the inequality in the second block is scaled \eqref{pair1}, the inequality in the third block is scaled \eqref{pair2}, and the inequalities in the fourth and fifth blocks are the three inequalities resulting from coefficient permutations of the scaled \eqref{dominate2}, with $t=0$ and $t=1,$ respectively. \hfill$\diamond$
\end{example}
\end{appendices}

\end{document}